\documentclass[a4paper, 12pt]{scrartcl}
\usepackage[ngerman, english]{babel}
\usepackage{amsfonts,amsmath, amssymb}
\usepackage{amsthm}
\usepackage{faktor}
\usepackage{xcolor}
\usepackage[arrow, matrix, curve]{xy}
\usepackage[left=3.5cm,right=3.5cm, bottom=4cm]{geometry}

\newenvironment{remark}[1][Remark]{\begin{trivlist}
\item[\hskip \labelsep {\bfseries #1}]}{\end{trivlist}}

\newtheorem{theorem}{Theorem}[section]
\newtheorem{lemma}[theorem]{Lemma}
\newtheorem{proposition}[theorem]{Proposition}
\newtheorem{corollary}[theorem]{Corollary}
\newtheorem{definition}[theorem]{Definition}
\newtheorem{conjecture}[theorem]{Conjecture}

\newcommand{\C}{\mathbb C}
\newcommand{\R}{\mathbb{R}}

\newcommand{\N}{\mathbb{N}}

\newcommand{\W}{\mathcal{W}}
\newcommand{\Szw}{\mathcal{S}^2}

\newcommand{\abs}[1]{\lvert#1\rvert}

\newcommand{\norm}[1]{\ensuremath{\left\|#1\right\|}}
\newcommand{\Sp}{\mathcal{S}}
\newcommand{\Ind}{\operatorname{Ind}_{\mathcal{W}}}

\newcommand{\Span}{\operatorname{span}}
\newcommand{\rank}{\operatorname{rank}}

\begin{document}

\title{On the Index of Willmore spheres}

\author{Jonas Hirsch\thanks{Universit\"at Leipzig, Fakult\"at f\"ur Mathematik und Informatik, Augustusplatz 10, 04109 Leipzig, hirsch@math.uni-leipzig.de}, Elena M\"ader-Baumdicker\thanks{Technische Universit\"at Darmstadt, Fachbereich Mathematik, Schlo\ss gartenstr.\ 7, 64289 Darmstadt, maeder-baumdicker@mathematik.tu-darmstadt.de}
\date{\vspace{-4ex}}}

\maketitle

\begin{abstract}
We consider unbranched Willmore surfaces in the Euclidean space that arise as inverted complete minimal surfaces with embedded planar ends. Several statements are proven about upper and lower bounds on the Morse Index -- the number of linearly independent variational directions that locally decrease the Willmore energy. We in particular compute the Index of a Willmore sphere in the three-space. This Index is $m-d$, where $m$ is the number of ends of the corresponding complete minimal surface and $d$ is the dimension of the span of the normals at the $m$-fold point. The dimension $d$ is either two or three. For $m=4$ we prove that $d=3$. In general, we show that there is a strong connection of the Morse Index to the number of logarithmically growing Jacobi fields on the corresponding minimal surface.

\end{abstract}

\section{Introduction}
For a closed, $2$-dimensional manifold $\Sigma$ and a smooth immersion $f:\Sigma \to (N,h)$ into a Riemannian manifold $(N,h)$, $\dim(N)\geq 3$, we define 
\begin{align*}
 \W(f) : = \frac{1}{4} \int_\Sigma |\vec H|^2 d\mu_g,
\end{align*}
where $g: = f^\# h$ is the induced metric on $\Sigma$ and  $\vec H = \sum_{i,j=1}^2 g^{ij} A_{ij}$ is the mean curvature vector, which is the trace of the second fundamental form $A_{ij} = (D_i \partial_j f)^\perp$ in local coordinates ($D_i$ is the covariant derivative along $f$). In Euclidean three-space $(N,h) = (\R^3, \delta_{\R^3})$ the mean curvature vector can be written as $\vec H = (\kappa_1 + \kappa_2) n$, where $\kappa_1, \kappa_2$ are the principle curvatures and $n$ is a unit normal at the given point (with the appropriate orientation). \\[-0.2cm]

In the 1920's the school of Blaschke \cite{Blaschke, Thomsen} showed that for surfaces into the Euclidean space $\W$ is invariant under scaling, rigid motions and under inversions at spheres (such as $ x\mapsto \frac{x}{|x|^2}$), where the center of the inversion is not on the surface $f(\Sigma)$.

Willmore rediscovered that functional in the 1960's \cite{Willmore}
and many authors use his name for it: the \emph{Willmore functional}. For more information about topics related to the Willmore functional see for example \cite{KuwertSch, MarquesNeves, Riviere2014} and the references therein.\\[-0.2cm]

Using the conformal invariance of the Willmore functional we get the following fact: Take a complete minimal surface $M^2$ in $\R^n$, $n\geq 3$, and choose a point $x_0\not \in M$. By inverting $M$ at the sphere $\partial B_1(x_0)$ we get a compact surface in $\R^n$ with one point removed. This surface satisfies the \emph{Willmore equation} (the Euler Lagrange equation for $\W$) except at the center of the inversion. If the immersion can be extended smoothly across that ``singularity'' then the surface is a smooth Willmore immersion\footnote{This is not always the case as Kuwert and Sch\"atzle showed in \cite{Kuwert2004}: An inverted catenoid cannot be extended across the singularity. There is also other examples: The inverted Enneper surface is a closed branched Willmore surfaces, i.e.\ the parametrization is smooth across the point, but it has a branch point of order three.}. Bryant proved in \cite[p.~47]{BryantDuality} that minimal surfaces in $\R^3$ with embedded planar ends (i.e.\ each end looks like a single plane) are - after inversion at a sphere - smooth Willmore surfaces, thus the singularity can be removed in that case. Furthermore, all critical points of the Willmore energy of the type of a sphere in $\R^3$ arise as inversions of minimal spheres and their Willmore energy is $4\pi m$ for an $m\in \N$. This result is not true for surfaces of higher genus: Every stereographic projection of the Clifford torus ($\mathcal{S}^1({\frac{1}{\sqrt{2}}}) \times \mathcal{S}^1 ({\frac{1}{\sqrt{2}}}) \subset \mathcal{S}^3$) is a smooth Willmore surface - it is in fact the minimizer of $\W$ among competitors of genus $g\geq 1$ in $\R^3$ \cite{MarquesNevesWillmore} - but it is not an inverted minimal surface in $\R^3$.\\[-0.2cm]

Coming back to the result of Bryant \cite{BryantDuality} it is interesting that not all numbers $m\in \N$ can be realized as $\W(f_{\mathcal{S}^2}) = 4\pi m$ for a smooth closed Willmore immersion $f:\mathcal{S}^2\to \R^3$. There are no Willmore spheres for $m=2,3,5,7$ \cite{Bryant}. But there are Willmore spheres with $\W(f_{\mathcal{S}^2}) =   8\pi l$ for all $l\in \N\setminus\{1\}$ \cite{Bryant}. In the recent work \cite{Michelat9}, Michelat exludes also the case $m=9$. He conjectures that the number of ends must always be even. \\
It follows from Bryant's classification that after the round sphere (which has $\W = 4\pi$) the next possible energy level is $16\pi$ - which corresponds to $4$ embedded planar ends of the immersed minimal surface before the inversion. In \cite[Section~5]{Bryant}, Bryant analyzes how many Willmore spheres (up to conformal transformation of the ambient space) with $16\pi$ energy exist. It turns out that there is a family of four real parameters of them. The moduli space of these surfaces is also studied in the preprint \cite{KusnerSchmitt} by Kusner and Schmitt. Although Bryant has studied all of the four-ends spheres in \cite[Section~5]{Bryant}, there are at least two examples that appear in another context: The Rosenberg-Toubiana surface \cite{RosenbergToubiana} and the Morin surface \cite{Morin78}. An explicit parametrization of the minimal surface that gives rise to the Morin surface after inversion was found by Kusner \cite{Kusner}. The Morin surface has a $4$-fold orientation reversing symmetry. This means (after an appropriate rotation in $\R^3$) one can rotate the surface around the $x_3$-axis by $\frac{\pi}{2}$ and obtains the same surface back, but with the opposite orientation. \\[-0.2cm]

In this article we study the \emph{Morse Index} of closed Willmore immersions in $\R^n$ that are inverted minimal surfaces with embedded planar ends. The Morse Index (or Willmore Morse Index) is the number of linearly independent directions of a Willmore surface that locally produce a negative second variation of $\W$. Obviously, minimizers for a fixed genus (as the Clifford torus among surfaces in $\R^3$ with positive genus \cite{MarquesNevesWillmore}) or minimizers in regular homotopy classes \cite{WeinerHomotopy, we} have Morse Index zero, there is no variational direction that decreases their Willmore energy. Two immersions $f,g:\Sigma \to\R^3$ are regulary homotopic if there is a homotopy of immersions $f:\Sigma\times [0,1] \to\R^3$ such that $f(\cdot,0)=f$,$f(\cdot, 1)=g$ and which induces a homotopy of the tangent bundles. As far as we know it is not known that one of the examples of Willmore surfaces that appear in the the literature has positive Morse Index. But it was conjectured in \cite[Section~3]{Francis95} that the Morin surface has Morse Index one. It is one of the results of this article to prove this conjecture. Furthermore, as the Willmore spheres in $\R^3$ get more and more complicated the more ends the corresponding surfaces have it is natural to believe that the Morse Index increases when the number of ends increases. This is in analogy of the theory of the Morse Index of the Area functional of minimal surfaces such as in  \cite{Otis}.\\[-0.2cm]

Before explaining our main results we would like to mention that the Morin surface is an important surface for the theory of the \emph{Sphere Eversion}. The fact that a standard round sphere can be turned inside out is called Sphere Eversion. More precisely, the round sphere in $\R^3$ with a given orientation is regularly homotopic to the round sphere with the opposite orientation. The remarkable statement of the existence of a Sphere Eversion was proven by Smale \cite{Smale}. Since then, many mathematicians contributed to this field \cite{Morin78, FrancisMorin, Apery, Morin78b, Morin78c, BanchoffMax, Riviere2018, Riviere2017, Michelat2018}. We do not want to go into details at this point but we would like to mention that the Morin surface is used as a  ``half-way model'' of the Sphere Eversion. The work of \cite{Francis95, Francis97} suggests that an \emph{optimal Sphere Eversion} can be constructed by the Morin surface. It is optimal in the sense that at every time of the deformation of the sphere the immersion has the least possible Willmore energy (thus bending energy). Having Morse Index at least one is a necessary condition for the Morin surface to be the half-way model for an optimal sphere eversion.\\[-0.2cm]

Our main results are the following:
\begin{theorem} \label{thm:Indexspheres}
 Let $\Psi:\Szw \to\R^3$ be an unbranched Willmore sphere. Then there is an $m\geq 1$ with $ \W(\Psi) = 4\pi m$ \cite{BryantDuality}. If $m>1$, then the Morse Index is
 \begin{align*}
  \Ind\left(\Psi\right) = m -\dim\operatorname{span}\{n_\Psi(p_i): i =1,...,m\}.
 \end{align*}
 \end{theorem}
 
 For $m=1$, $\Psi$ is the round sphere which has Index zero. For arbitrary $m>1$ we will show 
 \begin{align*}
  2 \leq d:= \dim\operatorname{span}\{n_\Psi(p_i): i =1,...,m\} \leq 3,
 \end{align*}
see Proposition~\ref{prop:allcodim2}. For $m=4$ we will show
\begin{proposition} \label{prop:dimthree}
 Let $\Psi:\Szw \to\R^3$ be a smooth Willmore sphere which arises as the inversion of a complete minimal sphere with four embedded planar ends then the normals at the ends $\{n_\Psi(p_i): i= 1,...,4\}$ span the whole $\R^3$, i.e.\ $d=3$.
\end{proposition}
We think that this is always the case. Note that $\dim\operatorname{span}\{n_\Psi(p_i): i =1,...,m\} = \dim\operatorname{span}\{n_X(p_i): i =1,...,m\}$, where $X$ denotes the complete minimal surface after inverting at the $m$-fold point.  Our conjecture is the following:
\begin{conjecture} \label{conj}
  Let $X:\Sigma\setminus\{p_1,...,p_m\} \to\R^3$ be a complete minimal surface with $m$ embedded planar ends $\{p_1,...,p_m\}$ then the normals at the ends $\{n_X(p_i): i= 1,...,m\}$ span $\R^3$, i.e.\ $d=3$.
\end{conjecture}

Theorem~\ref{thm:Indexspheres} and Proposition~\ref{prop:dimthree} implies
\begin{theorem}
 Let $\Psi:\Szw \to\R^3$ be an unbranched Willmore sphere with $16\pi $ Willmore energy. Then the Willmore Index is one.
\end{theorem}

Another statement that follows from the results above:
\begin{theorem}
 All unbranched Willmore spheres in $\R^3$ except of the round sphere are unstable.
\end{theorem}

 For other topological types of $\Sigma$ and higher codimension we will prove upper bounds on the Index of the following kind:

\begin{theorem} \label{thm:upper bounds}
  Let $\Psi:\Sigma \to\R^n$, $n\geq 3$, be a closed Willmore immersion such that $X:= \frac{\Psi}{|\Psi|^2}:\Sigma\setminus\{p_1,..., p_m\} \to \R^n$ is a complete, immersed minimal surface with $m$ embedded planar ends, $m > 1$. \\
 Then the Morse Index of $\Psi$ is bounded in the following way:
\begin{align*}
 \Ind\left(\Psi\right) \leq km -k -1,
\end{align*}
where $k:= n-2$ is the codimension of the surfaces.
\end{theorem}
In the codimension one case the upper bound in Theorem~\ref{thm:Indexspheres} also holds for other topological types:
\begin{proposition}
 Let $\Psi:\Sigma \to\R^3$ be a closed Willmore immersion such that $X:= \frac{\Psi}{|\Psi|^2}:\Sigma\setminus\{p_1,..., p_m\} \to \R^3$ is a complete, immersed minimal surface with $m$ embedded planar ends, $m> 1$. Let $n_{\Psi}$ be a unit normal vector field along $\Psi$. Denote by 
$
  d:= \dim\operatorname{span}\{n_{\Psi}(p_1),...,n_{\Psi}(p_{m})\}$
the dimension of the normals at the $m$-fold point of the Willmore surface $\Psi$. Then we have that
\begin{align*}
 \Ind\left(\Psi\right) \leq m-d.
\end{align*}
 \end{proposition}

For lower bounds on the Morse Index we study certain unbounded Jacobi fields of a minimal surface with embedded planar ends. One result is the following:
\begin{theorem}
Let $\Psi:\Sigma \to\R^3$ be a closed Willmore immersion such that $X:= \frac{\Psi}{|\Psi|^2}:\Sigma\setminus\{p_1,..., p_m\} \to \R^3$ is a complete, immersed minimal surface with $m$ embedded planar ends, $m> 1$. \\
Assume that there exists a logarithmically growing Jacobi field on $X$, i.e.\ a function $ u\in C^{2,\alpha}(\Sigma\setminus\{p_1,...,p_m\}) $, $L u = 0$, with expansion
$$u = \beta_i \log |z| + \tilde u_i (z) \ \text{ with } \ \tilde u_i \in C^{2,\alpha}(\overline{B_\epsilon}), \ \beta_i \in \R,  \  \sum_{i=1}^m \beta_i \not = 0.$$
at each end $p_i$ in local conformal coordinates $z$. \\
Then we have that $\Ind(\Psi) \geq 1$. 
\end{theorem}
We gain the optimal lower bounds in the proof of Theorem~\ref{thm:Indexspheres} for spheres. 
For the above mentioned Morin surface \cite{Morin78, Kusner} we study the situation more carefully. It turns out that it is stable under variations that preserve the $4$-fold orientation reversing symmetry, see Corollary~\ref{cor:Morin}. And the variation that locally decrease the Willmore energy has a two-fold symmetry, see Proposition~\ref{prop:MorinSymmetry}. Both statements about the Morin surface were conjectured in \cite{Francis95, Francis97}.\\[-0.2cm]

We want to mention that $\Ind\left(\Psi\right) \leq m$ in the codimension one case was already proven by Alexis Michelat \cite{Michelat} with different methods. Furthermore, Michelat proposes in \cite{MichelatBranched} a renormalised energy expansion to compute the Index. Using this expansion he is able to reduce the Index question to the determination of the eigenvalues of a finite dimensional matrix $A$. He is able to give an upper bound  $m-1$ for the Index in codimension one. This approach is applicable also to branched Willmore spheres. \\[-0.2cm]

This article is structured as follows: In Section~\ref{Section2} we recall the main definitions concerning the Morse Index and prove the upper bounds. Also the stability of the Morin surface under variations that preserve the given symmetry can be found here. Section~\ref{section3} contains the computation of the second variation of the Willmore functional for functions that are not smooth but satisfy a certain expansion at the ends (after inversion). In this section we alse re-prove the formula for second variation of $\W$ proven by Michelat in \cite{Michelat} for smooth variations. We also explain how logarithmically growing Jacobi fields on the minimal surface with a certain condition produce positive Index on the Willmore surface.  Note that these unbounded Jacobi field were already studied in the work of P\'{e}rez-Ros \cite{PerezRos} about the moduli space of complete minimal surfaces. In Section~\ref{section4} we combine and generalize ideas of P\'{e}rez-Ros \cite{PerezRos} and Montiel-Ros \cite{MontielRos} to prove higher positive Index for spheres. And we show that the variation that produces Index for the Morin surface has a two-fold symmetry.
 
\section*{Acknowledgment}

We would like to thank Tobias Lamm and Ben Sharp for rising questions about the Morse Index of the Morin surface.\\
The second author would like to thank Fernando Marques, Casey Kelleher, Otis Chodosh, Tobias Lamm and Ernst Kuwert for their supportive interest in her work and for useful discussions.
Furthermore, the second author would like to thank Laurent Hauswirth and Joaqu\'{\i}n P\'{e}rez for very helpful discussions and reference hints.
The second author is funded by the DFG (MA 7559/1-1) and would like to thank the DFG for the support. Part of the work was done during a research stay of the second author at the Princeton University funded by the Leopoldina German National Academy of Sciences. The hospitality of the Princeton University and the generosity of the Leopoldina foundation is greatly appreciated.

\section{Upper bounds} \label{Section2}
In this section we prove upper bounds for closed Willmore immersions in $\R^n$, $n\geq 3$, coming from a complete minimal surface by inversion at a sphere. The key ingredient is the Li-Yau inequality from~\cite{LiYau}.\\
From now on $\Sigma$ denotes an abstract, two-dimensional closed manifold.
% Second Variation of Willmore: For Minimal surfaces in S^3: Weiner, fuer minimal surfaces in S^n: Ndiaye und Schaetzle, fur beliebige Flaechen in N^3: Lamm-Metzger-Schulze, fuer beliebige Flaechen in N^n: Michelat??

\begin{definition}\label{def:index}
Let $\Psi: \Sigma\to \R^n$ be a smooth Willmore immersion and $\hat g: = \Psi^\# \delta$.  Let $\Gamma(N\Sigma)$ be the smooth sections of the normal bundle of $\Psi:\Sigma\to\R^n$. Then there is a strongly elliptic, $L^2$-selfadjoint operator $Z: \Gamma(N\Sigma) \to \Gamma(N\Sigma)$ of fourth order such that 
\begin{align*}
 \delta^2 \W(\Psi)(\vec v, \vec v): = \int_{\Sigma}  \vec v\cdot Z\vec v\, d\mu_{\hat g},
\end{align*}
where $Z$ is characterized by the property
\begin{align}
 \frac{d^2}{d t^2} \W (\Psi_t)\big|_{t=0} = \int_{\Sigma}  \vec v\cdot Z\vec v\, d\mu_{\hat g},
\end{align}
for a \textbf{smooth} variation $\Psi_t:\Sigma\to \R^n$, $t\in (-\epsilon, \epsilon)$, i.e.\ $\Psi_t$ is a smooth immersion for $t\in(-\epsilon, \epsilon)$ and $\Psi_0 = \Psi$. Here, $\vec v :=\frac{d}{dt}\Psi_t\big|_{t=0}$ is the variational vector field. By composing with a tangential diffeomorphism we can always arrange that $\vec v\in \Gamma(N\Sigma)$, i.e.\ $\vec v^T =0$. The object $\delta^2 \W(\Psi)(\vec v, \vec v)$ is called \emph{Index form} and is well-defined for sections $\vec v$ that have regularity $W^{2,2}$ with respect to the metric $\hat g = \Psi^\# \delta$.\\
By the theory of strongly elliptic operators $Z$ can be diagonalized on $\Gamma(N\Sigma)$ with eigenvalues
$$\lambda_0 \leq \lambda_1\leq ... \leq \lambda_k\leq ...$$
The eigenspaces $Y_{\lambda_j}$ corresponding to $\lambda_j$ are finite dimensional.\\
The \emph{$\W$-Index} and the $\W$-nullity of $\Psi$ are defined as follows:
\begin{align*}
 \Ind(\Psi)& : = \dim \left(\otimes_{\lambda<0} Y_{\lambda}\right),\\
 \operatorname{nullity}_{\W}(\Psi)& : = \dim \left(Y_0 \right).
\end{align*}
A Willmore surface is called \emph{stable} if $\Ind(\Psi) = 0$. This is equivalent to $\delta^2 \W(\vec v, \vec v) \geq 0$ for all $\vec v \in \Gamma(N\Sigma)$.

  \end{definition}
  
  \begin{remark}
   \begin{enumerate}
    \item The operator $Z$ has the form $Z \vec v = L L \vec v+ \Omega \vec v$, where $\Omega:\Gamma(N\Sigma) \to\Gamma(N\Sigma)$ is an operator of order two
    and $L:\Gamma(N\Sigma) \to\Gamma(N\Sigma)$ is the Jacobi operator of the area functional in $\R^n$, i.e.\ $$L\vec v = \Delta^\perp \vec v + |A|^2 \vec v.$$ The operator $\Delta^\perp$ is the Laplacian on the normal bundle and $A$ is the second fundamental form of $\Psi:\Sigma \to\R^n$, see for example \cite{Colding}.\\
    As $\left(\Delta^\perp\right)^2$ is the leading part of $Z$ and $\Sigma$ is compact, we see that $Z$ is strongly elliptic. 
    \item An explicit formula for $Z$ for a two-sided Willmore immersion $\Psi:\Sigma \to N^3$ into a smooth $3$-dimensional manifold $N$ is computed in \cite{Lamm1}. In this case, the normal bundle is trivial and any normal variation is of the form $\vec v = \alpha n$, where $n$ is a unit normal along $\Psi$. The Jacobi operator for the area functional then reads $L\alpha = \Delta \alpha + \alpha |A|^2 + \alpha \operatorname{Ric}^N(n,n)$ for the Laplace Beltrami operator $\Delta$ and the Ricci tensor of $N$, called $ \operatorname{Ric}^N$. The formula for $Z$ can be found in \cite[(2.14)]{Lamm1}:
    \begin{align*}
     Z \alpha & = LL \alpha - \frac{1}{2} H^2 L \alpha + 2 H\langle \mathring{A}, \nabla^2 \alpha\rangle + 2 H \operatorname{Ric}^N( n,\nabla \alpha)^T + 2\mathring{A}(\nabla \alpha, \nabla H)  \\
     &  \quad\quad + \alpha \left(|\nabla H|^2 + 2 \operatorname{Ric}^N(n, \nabla H)^T + H\Delta H + 2 \langle \nabla^2 H, \mathring{A} \rangle \right. \\ 
     & \quad \quad\ \ \ \ \left.+ 2H^2 |\mathring{A}|^2 + 2 H\langle \mathring{A}, T\rangle - H \nabla_n  \operatorname{Ric}^N(n,n) \right).    \end{align*}
Note the different sign convention for $L$ in \cite{Lamm1}. Here, $H$ is the (scalar) mean curvature, $T_{ij} = \operatorname{Rm}^N (\partial_i, n ,n, \partial_j)$ and $\mathring{A}$ is the tracefree second fundamental form.
\item In \cite{Michelat}, Alexis Michelat computed the second variation for a Willmore immersion $\Psi:\Sigma \to N^n$ into a general $n$-dimensional manifold $N^n$, see formula $(2.11)$ in \cite{Michelat} and the remarks below that formula. Here, less regularity is required for the immersions $\Psi_t$, $t\in (-\epsilon, \epsilon)$. 
\item For a minimal immersion $\Psi:\Sigma \to \Sp^n$ the second variation was computed in \cite{Weiner} by Weiner for $n=3$ and in \cite{Ndiaye} by Ndiaye and Sch\"atzle for general $n$.
   \end{enumerate}

  \end{remark}

\begin{theorem} \label{thm:leqm}
 Let $\Psi:\Sigma \to\R^n$, $n\geq 3$, be a closed Willmore immersion such that $X:= \frac{\Psi}{|\Psi|^2}:\Sigma\setminus\{p_1,..., p_m\} \to \R^n$ is a complete, immersed minimal surface with $m$ embedded planar ends, $m\geq 1$. Then the Willmore index is bounded above by the number of ends times the codimension:
 \begin{align*}
  \Ind \left(\Psi\right)\leq k m,
 \end{align*}
where $k := n-2$.
\end{theorem}
\begin{remark}\begin{enumerate}
               \item  In \cite{Michelat}, Alexis Michelat gave a prove of the theorem above in codimension one  using an explicit formula for the second variation. That formula will play an important role later. But it is not needed in our proof here.
               \item Under the condition of the Theorem it follows from \cite[Lemma~1]{Kusner} that $\Psi$ has Willmore energy $4\pi m$ and that there is a point $x\in \Psi(\Sigma)$ with multiplicity $m$. In the given setting we have that $x=0$. \\
               In fact, a Willmore surface $\Psi$ has a point of multiplicity $m$ and $4\pi m$ energy if and only if $\Psi$ arises as the inversion of a complete minimal surface with $m$ embedded planar ends.\\
               It also follows from the condition of the Theorem that $\Psi$ does not have another point of multiplicity $m$. This is proven in \cite[Lemma~2]{Kusner} and can be seen with different methods in \cite[p.~1191]{Lamm2}.
              \end{enumerate}
\end{remark}

\begin{proof}
Note that $\Psi$ has a point of multiplicity $m$ with preimages $p_1,...,p_m$ and $4\pi m$ Willmore energy, see the remark above.\\
Let $N\Sigma=\bigcup_{p\in \Sigma} N_p\Sigma$ be the normal bundle of the closed immmersion $\Psi$. Denote by $\Gamma (N\Sigma)$ the smooth sections of the normal bundle. We define  
$$V_0: =\{\vec v \in \Gamma (N\Sigma): \vec v(p_i) = 0 \ \forall i=1,...,m\}.$$
For $\vec v\in V_0$ we define $\Psi_t: = \Psi + t \vec v $ for $t\in(-\epsilon, \epsilon)$, $\epsilon$ small but fixed. By the definition of $V_0$ we get that $\Psi_t(p_i) = \Psi_t(p_j)$ for $i\neq j$. This implies that $\Psi_t$ still has a point of multiplicity $m$. By Li and Yau \cite{LiYau} we know that the existence of a point of multiplicity $m$ implies at least $4\pi m$ amount of Willmore energy. This means that the Willmore energy cannot decrease for a variation of the form $\frac{d}{d t}\big|_{t=0}\Psi =\vec v\in V_0$. We get that
 \begin{align*}
  \delta^2\W(\Psi)(\vec v,\vec v) = \frac{d^2}{d t^2}\big|_{t=0}\W(\Psi_t) \geq 0,
 \end{align*} 
which implies $V_0\subset \{\vec v: \delta^2\W(\Psi)(\vec v,\vec v)\geq 0\}$. The set $V_0$ is a vector space of codimension at most $k m$ which can be seen by $V_0 = \operatorname{kernel} A$, where $A:\Gamma(N\Sigma) \to \otimes_{i=1}^m N_{p_i}\Sigma \simeq \R^{m\times k}$, $A(\vec v)=(\vec v(p_1),...,\vec v(p_m))$ is a linear map. By the definition of the Index we get that $\Ind(\Psi)\leq km$.
%  As $X$ has $m$ embedded planar ends and is minimal, the Willmore surface $\Psi$ has Willmore energy $4\pi m$ and a point of multiplicity $m$ in the image.  Define $$V_0: =\{v \in W^{2,2}(\Sigma, \R)\cap W^{1,\infty}(\Sigma, \R): v(p_i) = 0 \ \forall i=1,...,m\}.$$
%  For $\epsilon$ small enough
% % \\
% !!! The proof is only codimension one!!-------------------------
%  As $X$ has $m$ embedded planar ends and is minimal, the Willmore surface $\Psi$ has Willmore energy $4\pi m$ and a point of multiplicity $m$ in the image.  Define $$V_0: =\{v \in W^{2,2}(\Sigma, \R)\cap W^{1,\infty}(\Sigma, \R): v(p_i) = 0 \ \forall i=1,...,m\}.$$ Let $n_{\Psi}$ be a unit normal vector field along $\Psi$. For $v\in V_0$ we define $\Psi_t: = \Psi + t v n_\Psi$ for $t\in(-\epsilon, \epsilon)$, $\epsilon$ small but fixed. We notice $\frac{d}{dt}\big|_{t=0}\Psi_t = 0$ by the definition of $V_0$. This implies that the variation $\frac{d}{dt}\big|_{t=0}\Psi_t = vn_{\Psi}$ does not move the point of multiplicity $m$. By Li and Yau \cite{LiYau} we know that the existence of a point of multiplicity $m$ implies at least $4\pi m$ amount of Willmore energy. This means that the Willmore energy cannot decrease for a variation of the form $v n_{\Psi}$, $v\in V_0$. We get that
%  \begin{align*}
%   \delta^2\W(\Psi)(v,v) = \frac{d^2}{d t^2}\big|_{t=0}\W(\Psi_t) \geq 0,
%  \end{align*} 
% which implies $V_0\subset \{v: \delta^2\W(\Psi)(v,v)\geq 0\}$. The set $V_0$ is a vector space of codimension at most $m$. By the definition of the index we get that $\Ind(\Psi)\leq m$.
\end{proof}
\begin{remark}
 The property $\{v \in C^\infty(\Sigma): v(p_i) = 0\ \forall i=1,...,m\}\subset \{v: \delta^2\W(\Psi)(vn_{\Psi},vn_\Psi)\geq 0\}$ was already shown by Michelat in \cite{Michelat} with other methods for the codimension one case. 
\end{remark}

The result of Theorem~\ref{thm:leqm} can be improved in the following way:

\begin{theorem}\label{thm:allcodim}
 Let $\Psi:\Sigma \to\R^n$, $n\geq 3$, be a closed Willmore immersion such that $X:= \frac{\Psi}{|\Psi|^2}:\Sigma\setminus\{p_1,..., p_m\} \to \R^n$ is a complete, immersed minimal surface with $m$ embedded planar ends, $m > 1$. \\
 For every point $p_i$ consider the decomposition $\R^n = T_{\Psi(p_i)}\Sigma \oplus N_{\Psi(p_i)}\Sigma $ into the tangent and normal space. We define the linear map 
 $$J: \R^n \to \otimes_{i=1}^m N_{\Psi(p_i)}\Sigma \simeq \R^{mk}, \ \ J(v): = (\pi_1(v), ... ,\pi_m(v)),$$ where $\pi_i: \R^n \to N_{\Psi(p_i)}\Sigma$ is the projection of a vector along $\Psi$ at $p_i$ onto the normal space at that point. Then the Index of $\Psi$ is bounded in the following way:
\begin{align*}
 \Ind\left(\Psi\right) \leq km -\rank J,
\end{align*}
where $k:= n-2$ is the codimension of the surfaces.
\end{theorem}
\begin{remark}
	The assumption $m>1$ is not really a restriction. Since a complete, immersed minimal surface with a single embedded planar end is a plane, the corresponding Willmore surface ist the absolute Willmore minimizer, the round sphere. 
\end{remark}

\begin{proof}
 As the Willmore functional is invariant under translations, each vector $a\in\R^n$ corresponds to a $\W$-Jacobi field $a^\perp \in \Gamma(N\Sigma)$, where $(\cdot)^\perp:\R^n \to N\Sigma$ is the projection onto the normal bundle.\\
 We use the notation $I:\otimes_{i=1}^m N_{\Psi(p_i)}\Sigma \to \R^{mk} $ for the isomorphism identifying $\otimes_{i=1}^m N_{\Psi(p_i)}\Sigma$ with $\R^{mk}$.  Let $d:=\rank J$ and arrange the matrix corresponding to the linear map $I\circ J$ in such a way that the first $d$ lines of $I\circ J$ are linearly independent. We denote by $\hat J: \R^n \to \R^{d}$ the map $\hat J = \sigma \circ I \circ J$, where $\sigma:\R^{mk} \to \R^d$ is the projection onto the first $d$ components. The map $\hat J$ has full rank, which implies that $\sigma\circ I (\vec v(p_1),..., \vec v(p_m))$ has a preimage under $\hat J$. We denote this vector by $a$ and notice $a = \hat J^{-1} \circ \sigma \circ I \circ A(\vec v)$, where $A:\Gamma(N\Sigma) \to \otimes_{i=1}^m N_{\Psi(p_i)}\Sigma $ is the evaluation at the points $p_1,...,p_m$, $A(\vec v) = (\vec v(p_1),..., \vec v(p_m))$.\\
  Define $\vec w:= \vec v - a^\perp$. Then $\vec w(p_1)= ... =\vec w(p_d)=0$ by construction, and $a^\perp$ is a $\W$-Jacobi field. The symmetry of the second variation implies $\delta^2\W(\Psi)(\vec v, a^\perp)=0$ and thus
 \begin{align*}
  \delta^2\W(\Psi)(\vec v, \vec v) &=  \delta^2\W(\Psi)(\vec v -a^\perp, \vec v- a^\perp)=  \delta^2\W(\Psi)(\vec w, \vec w).
 \end{align*}
Denote by $\hat A: \Gamma(N\Sigma) \to \{0\}^{d} \times \otimes_{i=d+1}^m N_{p_i}\Sigma \simeq \{0\}^{d} \times \R^{km  - d}$ the linear map $$\vec v\mapsto  (0,...,0,\vec w(p_{d+1}),..., \vec w(p_m)).$$ As in the proof of Theorem~\ref{thm:leqm} we know that $\operatorname{kernel} \hat  A \subset \{\vec v: \delta^2\W(\Psi)(\vec v, \vec v)\geq 0\}$. Since $\hat A$ is a linear map, the cokernel has at most dimension $km-d$, which implies $\Ind (\Psi)\leq km-d$.
\end{proof}

For the sake of simplicity we state the result of Theorem~\ref{thm:allcodim} for codimension one.

\begin{corollary}\label{cor:minusd}
 Let $\Psi:\Sigma \to\R^3$ be a closed Willmore immersion such that $X:= \frac{\Psi}{|\Psi|^2}:\Sigma\setminus\{p_1,..., p_m\} \to \R^3$ is a complete, immersed minimal surface with $m$ embedded planar ends, $m> 1$. Let $n_{\Psi}$ be a unit normal vector field along $\Psi$. Denote by 
 \begin{align*}
  d:= \dim\operatorname{span}\{n_{\Psi}(p_1),...,n_{\Psi}(p_{m})\}
 \end{align*}
the dimension of the normals at the $m$-fold point of the Willmore surface $\Psi$. Then we have that
\begin{align*}
 \Ind\left(\Psi\right) \leq m-d.
\end{align*}

\end{corollary}
\begin{proof}
 Note that $ \dim\operatorname{span}\{n_{\Psi}(p_1),...,n_{\Psi}(p_{m})\} = \rank J$ for the linear map $J$ from Theorem~\ref{thm:allcodim}.
\end{proof}

\begin{proposition}\label{prop:allcodim2}
 Let $\Psi:\Sigma \to\R^n$, $n\geq 3$, be a closed Willmore immersion such that $X:= \frac{\Psi}{|\Psi|^2}:\Sigma\setminus\{p_1,..., p_m\} \to \R^n$ is a complete, immersed minimal surface with $m$ embedded planar ends, $m > 1$. \\
 As in Theorem~\ref{thm:allcodim} we define the linear map 
 $J(v) = (\pi_1(v), ... ,\pi_m(v)),$ where $\pi_i: \R^n \to N_{\Psi(p_i)}\Sigma$ is the projection of a vector along $\Psi$ at $p_i$ to the normal space at that point. Then we have that 
 \begin{equation}\label{eq:rank J> n-2} \rank J\geq k+1,\end{equation}
 where $k=n-2$ is the codimension of the surfaces. As a consequence we get that
 \begin{align}
  \Ind(\Psi) \leq km -k-1.
 \end{align}
\end{proposition}

\begin{proof}
Since $\pi_1: \R^n \to N_{\Psi(p_1)}\Sigma \cong \R^{n-2}$ is an orthogonal projection we clearly have 
\[ n-2 \le \rank J \le n. \]

Before we are going to show \eqref{eq:rank J> n-2} by induction on $n\ge3$ let us note that since $i(x):=\frac{x}{\abs{x}^2} : \R^n \to \R^n$ is conformal and $X=i\circ \Psi$ we deduce that the normal space of $X$ at one of the embedded planar end $p_i$ we have that
\[ N_{X(p_i)}\Sigma = N_{\Psi(p_i)}\Sigma. \]
Now assume by contradiction that $rank J=n-2$ i.e.\ $N_{\Psi(p_i)}\Sigma=N_{\Psi(p_j)}\Sigma$ for all $i, j$. Since they all agree and $\dim N_{\Psi(p_i)}\Sigma \ge 1$ we may fix $v \in N_{\Psi(p_i)}\Sigma$ with $\abs{v}=1$. As the coordinate functions of $X$ are harmonic, we conclude that 
\[ u(x):= v \cdot X(x)\]
is a harmonic function on $\Sigma \setminus \{ p_1, \dotsc, p_m \} $. But since we have chosen $v\in N_{X(p_i)}\Sigma$ and each end $p_i$ is planar, we deduce that $u$ is actually bounded. This implies that $u$ is constant, $u \equiv u_0$, compare Lemma~\ref{lem:bounded harmonic functions are constant}. In particular, we have $X(\Sigma) \subset \{ x \in \R^n \colon v\cdot x = u_0 \} \cong \R^{n-1}$. If $n=3$ we have that $X$ is a plane, which contradicts the assumption $m>1$. For $n>3$ we may translate $X$ if necessary to ensure that $u_0$ is $0$. But since $ \{ x \in \R^n \colon v\cdot x = 0 \} $ is invariant under $i$ we deduce that $\Psi: \Sigma \to \{ x \in \R^n \colon v\cdot x = 0 \} \cong \R^{n-1}$. This as well a contradiction by induction. 
\end{proof}

For a surface of codimension one, Proposition~\ref{prop:allcodim2} obviously reads $\Ind(\Psi)\leq m-2$. It turns out that we can improve this result for spheres of codimension one with four embedded planar ends.

\begin{theorem}\label{thm:forspheres}
  All Willmore spheres $\Psi:\mathcal{S}^2\to\R^3$ that arise as an inversion of a complete minimal surfaces with four embedded planar ends satisfy $$\Ind(\Psi)\leq 1.$$
\end{theorem}

 For the proof of this theorem we need two lemmas.
 \begin{lemma} \label{lem:M}
 For $p_2,...,p_m \in \C$ we define the anti-symmetric matrix $M \in \C^{m\times m}$ by
\[ M_{ij}:= \begin{cases} \frac{1}{p_i - p_j} \text{ for } i \neq j\text{ and } i,j\geq 2\\
-1 \text{ for } i =1 , j>1\\ 
1  \text{ for } i>1 , j=1\\
0  \text{ for } i=j.
\end{cases} \ \ \ i,j=1,...,m. \]
If there are two linearly independent vectors $\vec a, \vec b$ in the kernel of $M$ such that $|a_i|= |b_i|$ for $ i=1,...,m$ then there is a complete minimal sphere in $\R^3$ with $m$ embedded planar ends where $\dim\Span\{n_{X}(p_1),..., n_{X}(p_m)\}=2$.
 \end{lemma}
 \begin{proof}
  Let $X:\Sp^2 \to\R^3$ be a complete minimally immersed sphere in $\R^3$ with $m$ embedded planar ends where $\dim\Span\{n_{X}(p_1),..., n_{X}(p_m)\}=2$.\\
   Without loss of generality we may assume that the minimal immersion $X$ is conformally parametrized, hence we have
\[ X(z) = \Re \left( \int \phi \, dz \right) \]
for some holomorphic function $\phi=(\phi_1, \phi_2, \phi_3): \hat{\C} \setminus\{p_1, \cdots, p_{m} \} \to \C^3$ satisfying $\phi^2 =0$ and $\phi(z) \neq 0\ \forall z \in \hat{\C}\setminus \{ p_1, \cdots, p_{m}\}$, (the last condition ensures that $X$ is an immersion). 
Precomposing $\phi$ with a M\"obius transformation of $\hat{\C}$ we may assume that $p_1 = + \infty$ (later $p_2=0$, $p_3=1$).

Recall that $X$ having a planar end in $p_i$ implies that locally we have
\[
\phi(p_i+z) = -\frac{v_i}{z^2} + h_i(z) 
\]
for some $v_i \in \C^3\setminus \{0\}, v_i^2 =0$ and some holomorphic $h_i(z): D_\delta \to \C^3$. In particular we deduce that
\begin{itemize}
\item[$(a_i)$] $\abs{(z-p_i)^2 \phi (z)}$ is bounded in a neighborhood of $p_i$ for all $i>1$;
\item[$(b_i)$] $ \lim_{z \to p_i} \partial_z \left( (z-p_i)^2 \phi (z) \right)= 0$ for all $i>1$.
\end{itemize}
Similarly we can state the conditions at $p_1=+\infty$. Observe that $z \mapsto X(\frac{1}{z})$ is a local conformal parametrization around $p_1$. Since in this parametrization $0$ is a planar end we must have that 
\[ 2\partial_z X(\frac{1}{z}) = - \frac{v_1}{z^2} + h_1(z) \]
for some $v_1 \in \C^3\setminus\{0\}, v_1^2=0$ and some holomorphic $h_1(z):D_\delta \to \C^3$. 
But since $2\partial_z X(\frac{1}{z}) = -\frac{\phi(\frac{1}{z})}{z^2}$
we deduce
\begin{itemize}
\item[$(a_1)$] $\abs{\phi (\frac{1}{z})}$ is bounded in a neighborhood of $0$, i.e.\ $\abs{\phi(z)}$ is bounded at $p_1=+\infty$.
\item[$(b_1)$] $ \lim_{z\to 0 } \partial_z \phi(\frac{1}{z}) =0$ i.e. $\lim_{z \to \infty} z^2 \phi(z) =0$.
\end{itemize}
In the following we want to translate the condition that the dimension of the normals at the ends is two and $ (a_i), (b_i)$, $i=1, \dotsc , m$ into conditions on the Weierstra\ss\ data. \\

After a rotation we may assume that the normal $n_{X}^i$ at the end $p_i$ is orthogonal to $e_3$. This implies that in the expansions of $\phi$ at $p_i$ we must have $(v_i)_3 \neq 0$ and $\lim_{z \to p_i} \phi_3(z) \neq 0$ for all $i$. 
Recall that we can write \[\phi = \left( \frac 12 ( g^{-1} - g ) \phi_3, \frac{i}{2}( g^{-1} + g ) \phi_3, \phi_3 \right)\] for some meromorphic $g$ on $\hat{\C}$. \\
It is convenient to introduce the following polynomials:
\[
 \varphi_1(z):=  \prod_{j=2}^m (z-p_j), \quad \varphi_i(z):= \prod_{\substack{j=2\\ j\neq i}}^m (z-p_j), \ \ \text{ for } i=1,...,m.
\]
Observe that the functions
\begin{align*} P_1&:=-\varphi_1^2 (\phi_1+ i \phi_2)= \varphi_1^2 g \phi_3\\
P_2&:= \varphi_1^2 (\phi_1- i \phi_2)= \varphi_1^2 g^{-1} \phi_3 \\
P_3&:= \varphi_1^2 \phi_3 \end{align*}
have removable singularities at $p_i$ for all $i>1$ by condition $(a_i)$ and grow at most like $\abs{z}^{2(m-1)}$ at $p_1$ by $(a_1)$. Hence $P_l$ is a polynomial of order $\le 2(m-1)$ (``$=$'' for $l=3$).
Furthermore note that $P_1P_2 = P_3^2$ hence we must have $P_1= a^2 c$, $P_2=b^2 c$, $P_3 = abc$ for some polynomials $a,b,c$. But since \[ \phi_1= \frac{1}{2\varphi_1^2}( P_2 - P_1),\quad \phi_2=\frac{i}{2\varphi_1^2}(P_2 + P_1),\quad  \phi_3 = \frac{1}{\varphi_1^2} abc \]
we deduce that $c=const.$ otherwise there would be a point with $\phi(z)=0$. 
Let us summarize that $P_1=a^2, P_2=b^2, P_3=ab$. As observed previously we must have that $P_3$ is a polynomial of order $2m-2$ hence $a,b$ must be polynomials of order $m-1$. And they cannot have zeros in any $p_i$ since otherwise $ (z-p_i)^2\phi_3(z) = \frac{ab}{\varphi_i^2}$ would converge to $0$ contradicting $(v_i)_3 \neq 0$. 
Now it is straight forward that the conditions $(b_i)$ and $(b_1)$ translate to 
\begin{itemize}
\item[$(b_i^*)$] $\left(\frac{a^2}{\varphi_i^2}\right)'(p_i) =0$ which is equivalent to 
\begin{equation*}\label{eq:condition for polynomials at pi} \frac{a'(p_i)}{\varphi_i(p_i)} - \frac{a(p_i)}{\varphi_i(p_i)} \frac{\varphi_i'(p_i)}{\varphi_i(p_i)}=0 \end{equation*}
\item[$(b_1^*)$] $\lim_{z \to \infty} z^2 \left(\frac{a^2}{\varphi_1^2}\right)'(z) =0$ which is equivalent to
\begin{equation*}\label{eq:condition for polynomials at p0}\lim_{z \to \infty} z^2 \left(\frac{a'(z)}{\varphi_1(z)} - \frac{a(z)}{\varphi_1(z)} \frac{\varphi_1'(z)}{\varphi_1(z)}\right)=0, \end{equation*}
\end{itemize}
analogously for the polynomial $b$.

Recall that the meromorphic function $g$ in the Weierstra\ss\ data coincides with the stereographic projection of the Gau\ss\ map. From the computation above we get that $g=\frac{a}{b}$. Since we assume that all ends have a normal perpendicular to $e_3$, we have that
\begin{itemize}
\item[$(c_i )$] $\abs{g(p_i)}=1$ or equivalently $\abs{a(p_i)}= \abs{b(p_i)}$ for all $i\ge2$;
\item[$(c_1) $] $\lim_{z \to \infty} \abs{g(z)} =1$ or equivalently $\lim_{z \to \infty} \frac{ \abs{a(z)}}{\abs{b(z)}}=1$.
\end{itemize}
Note that this is the first time we use that the dimension of the span of the normals at the ends is two-dimensional.\\[0.2cm]
It is convenient to expand the polynomials $a$,$b$ in the basis of $\varphi_i$, $i=1, \dotsc, m$ i.e.
\[ a= \sum_{i=2}^{m} a_i \varphi_i + a_1 \varphi_1, \quad b= \sum_{i=2}^{m} b_i \varphi_i + b_1 \varphi_1.
\]
One checks that 
\[ \varphi_1'(p_i) = \varphi_i(p_i), \quad \varphi_i'(p_j)= \begin{cases} \frac{ \varphi_j(p_j)}{p_j - p_i} &\text{ for } j \neq i \\
\varphi_i(p_i) \left(\sum_{l \neq i} \frac{1}{p_i - p_l}\right) & \text{ for } i = j. \end{cases}\]
Using this in ($b_i^*$) we obtain 
\begin{equation}\label{eq:condition for coefficients at pi}
0=\sum_{l \neq i} a_l \frac{\varphi_l'(p_i)}{\varphi_i(p_i)} = \sum_{l \neq i} \frac{a_l}{p_i-p_l} 
\end{equation}
To transform condition ($b_1^*$) observe that $\varphi_i$ for $i\ge 2$ are polynomials of order $m-2$ hence $\lim_{z \to \infty} \frac{z^2 \varphi_i'(z)}{\varphi_1(z)} = m-2$ and $\lim_{z \to \infty} \frac{z \varphi'_1(z)}{\varphi_1(z)} \frac{z \varphi_i(z)}{\varphi_1(z)} = m-1 $. Thus, ($b_1^*$) transforms to 
\begin{equation}\label{eq:condition for coefficients at p0}
0=\lim_{z \to \infty} \sum_{i\ge 2} a_i \left( \frac{z^2 \varphi_i'(z)}{\varphi_1(z)} - \frac{z \varphi'_1(z)}{\varphi_1(z)} \frac{z \varphi_i(z)}{\varphi_1(z)} \right)  = - \sum_{i \ge 2 } a_i.
\end{equation}
Comparing this to the statement of the lemma we recognize that the matrix $M$ will give us the Weierstra\ss\ representation with the desired properties.
 \end{proof}

\begin{lemma} \label{lem:M2}
 Let $M \in \C^{m\times m}$ be the matrix from Lemma~\ref{lem:M} corresponding to points $p_2,...,p_m\in\C$. Let $m$ be even and $\dim \operatorname{kernel}(M) = 2$. Choose the order of the points $p_2,...,p_m$ in such a way that 
 $$M=\begin{pmatrix}
A & C\\
- C^T & B                                                                                                                                                                                                                                                                                                                                                                                                                                                                                                                     \end{pmatrix},
$$
where $A\in \C^{m-2\times m-2}$ is invertible. Then we have the following representation of the kernel of $M$: 
\begin{align}\label{kernelformula}
\operatorname{kernel} (M)&=\operatorname{span} \left\{ 
 \begin{pmatrix}
  - A^{-1}C e_l \\
  e_l
 \end{pmatrix}:
l=1,2 \right\} =:\{v_l: l=1,2\}.
\end{align}
If there are $j_0\in\{1,...,m-2\}$ and $j_1\in\{1,...,m-2\}$ such that $ \bar v_1^{j_0} v_2^{j_0}\in \R\setminus\{0\}$ and $ \bar v_1^{j_1} v_2^{j_1}\in i\R\setminus\{0\}$ then there do not exist two linearly independent vectors $\vec a, \vec b$ in the kernel such that $|a_j|=|b_j|$ for all $j=1,...,m$.
\end{lemma}

\begin{proof}
 As the kernel is nontrivial, there is a vector $\begin{pmatrix}
 v\\w
 \end{pmatrix}\neq \vec 0
$ such that $A v + Cw=\vec 0$ and $-C^T v + B w=\vec 0$. Note that $w\neq \vec 0$ because then $v=0$ would follow from the invertibility of $A$. We use that $A$ is of full rank and get $v= -A^{-1} C w $ and thus $(C^T A^{-1} C + B)w=0$. The matrix $C^T A^{-1}C + B$ is a $2\times 2$ skew-symmetric matrix. Since this matrix has a nontrivial element in its kernel and as the eigenvalues of skew-symmetric matrices come in pairs, we get $B=-C^T A^{-1} C $.\\
We choose $w_l=e_l$, $l=1,2$. Then the following vectors form a basis of the kernel of $M$:
\begin{align*}
\operatorname{kernel} (M)&=\operatorname{span} \left\{ 
 \begin{pmatrix}
  - A^{-1}C e_l \\
  e_l
 \end{pmatrix}:
l=1,2 \right\} =:\{v_l: l=1,2\}
\end{align*}
as claimed in the statement of the lemma.\\
We want to compare two linearly independent vectors in the kernel of $M$. So let $\vec a= \alpha^1 v_1 + \alpha^2 v_2$ and $\vec b= \beta^1 v_1 + \beta^2 v_2$ be two such vectors in the kernel. The necessary condition in order to create a surface where the dimension of the span of the normals of the ends are two-dimensional are $|a_j|=|b_j|$ for $j=1,...,m$. The form of $v_l$ in the last two components yields $|\alpha^l| =|\beta^l|$ for $l=1,2$. By multiplying with constants we can assume that $\alpha^1,\beta^1 \in\R$ which implies $\alpha^1=\pm\beta^1$. Let us assume $\alpha^1=\beta^1$.  We compute for $j=1,...,m-2$
\begin{align*}
 |a_j|^2 & = |\alpha^1|^2 |v_1^j|^2 + |\alpha^2|^2|v_2^j|^2 + 2 \Re(\bar\alpha^1\alpha^2 \bar v_1^j v_2^j)\\
 &= |\beta^1|^2 |v_1^j|^2 + |\beta^2|^2|v_2^j|^2 + 2 \Re(\bar\beta^1\beta^2 \bar v_1^j v_2^j) +2 \Re\left(\left(\bar\alpha^1\alpha^2 - \bar\beta^1\beta^2\right) \bar v_1^j v_2^j\right)\\
 & = |b_j|^2+ 2 \alpha^1\Re\left(\left(\alpha^2 - \beta^2\right) \bar v_1^j v_2^j\right).
\end{align*}
This implies $\Re\left(\left(\alpha^2 - \beta^2\right) \bar v_1^j v_2^j\right)=0$. 
If for $j_0$ there in an entry $v_1^{j_0}=0$ or $v_2^{j_0}=0$ then this is satisfied. So we consider all $j\in\{1,...,m-2\}$ such that $v_1^{j}\neq 0$ and $v_2^{j}\neq0$. \\
 Computing 
 \begin{align*}
 0&=\Re\left(\left(\alpha^2 - \beta^2\right) \bar v_1^{j_0} v_2^{j_0}\right)\\
 &=  v_1^{j_0} v_2^{j_0} \Re\left(\alpha^2 - \beta^2\right)\text{ and}\\
 0&= \Re\left(\left(\alpha^2 - \beta^2\right) \bar v_1^{j_1} v_2^{j_1}\right)\\
 &= \Im \left( \bar v_1^{j_1} v_2^{j_1}\right) \Im \left(\alpha^2 - \beta^2\right)
\end{align*}
implies that $\alpha^2 = \beta^2$ which contradicts the fact that $\vec a\neq\vec b$ taking into account $\alpha^1 = \beta^1$.\\
If $\alpha^1=-\beta^1$, we redo the computation above and get $\alpha^2 = - \beta^2$. Again, this contradicts the assumption of $\vec a$ and $\vec b $ beeing linearly independent.
\end{proof}

We use this for the proof of Theorem~\ref{thm:forspheres}
\begin{proof} \textit{Theorem \ref{thm:forspheres}}\\
We already know from Proposition~\ref{prop:allcodim2} that the dimension of the normals at the ends is either two or three. We show that the dimension of the normals at the ends cannot be three. We assume that this dimension is two. As we invert the Willmore surface at $\Psi(p_i)$ we get that $n_{X}(p_i)= n_{\Psi}(p_i)$ for all $ i=1,...,m$. Therefore, we consider now a complete minimal immersion in $\R^3$ with $m$ embedded planar ends where $\dim\Span\{n_{X}(p_1),..., n_{X}(p_m)\}=2$ and we will get a contradiction.\\
By applying a conformal transformation of $\mathcal S^2$ we can assume that $p_2=0$ and $p_3=1$. The matrix $M$ from the lemmas above then has the following form
\begin{align*}
 M^4=\begin{pmatrix}
    0 & -1 & -1 & -1\\
    1& 0 & -1 & -\frac{1}{t}\\
    1 & 1& 0 & \frac{1}{1-t}\\
    1 & \frac{1}{t} & \frac{1}{t-1}& 0
   \end{pmatrix},
\end{align*}
where $t$ is a complex parameter that has to be determined. The determinant of $M^4$ can be computed by the Pfaffian:\\
\begin{align*}
 \det (M^4) = \left(\operatorname{Pfaff}(M^4)\right)^2 = \left( -\frac{1}{1-t} - \frac{1}{t} + 1\right)^2 = \frac{(t^2 - t + 1)^2}{t^2 (1-t)^2}.
\end{align*}
So the parameter $t$ is determined by the equation $t^2 - t + 1=0$ which has the two solutions $t_{1,2}=\frac{1}{2}\pm i \frac{\sqrt{3}}{2}$. As computed above, the kernel of $M$ is spanned by the following two vectors
\begin{align*}
 v_l= \begin{pmatrix}
       - A^{-1} C e_l\\
       e_l
      \end{pmatrix}, \text{ where } A=\begin{pmatrix}
      0 & -1\\
      1 & 0\end{pmatrix} \text{ and } C=\begin{pmatrix}
      -1& -1\\
      -1 & -\frac{1}{t}
      \end{pmatrix},
      \end{align*}
which are $v_1 = (1,-1,1,0)$ and $v_2=(\frac{1}{t},-1,0,1)$ with $t\in i\R$. This implies $\bar v_1^1 v_2^1 \in i\R\setminus\{0\}$ and $\bar v_1^2 v_2^2 \in \R\setminus\{0\}$. By Lemma~\ref{lem:M2} the claim follows.
\end{proof}

\begin{remark}\begin{itemize}
  \item There is another way of proving Theorem~\ref{thm:forspheres}: Montiel and Ros proved in \cite[Corollary~16]{MontielRos} that if the branching values of the ramification points of the meromorphic map $g:\Sigma\to \Sp^3$ (from the Weierstr\ss\ representation) lie in an equator of $\Sp^2$ then the minimal surface must have at least one logarithmic end. We have shown in the proof of Lemma~\ref{lem:M} that the meromorphic function $g$ has degree $d=m-1$ which is $3$ here. Therefore, there are $2d-2=4$ ramifcation points of $g$ and $4$ ends at $p_1,...,p_4$, see \cite{MontielRos}.
 From the condition $ \frac{a'(p_i)}{\varphi_i(p_i)} - \frac{a(p_i)}{\varphi_i(p_i)} \frac{\varphi_i'(p_i)}{\varphi_i(p_i)}=0$ (which is condition $(b_i^*)$) and the corresponding equation $ \frac{b'(p_i)}{\varphi_i(p_i)} - \frac{b(p_i)}{\varphi_i(p_i)} \frac{\varphi_i'(p_i)}{\varphi_i(p_i)}=0$ we see that $0 = a'(p_i) b(p_i) - a(p_i) b'(p_i)$. As $g = \frac{a}{b}$ we get that every $p_i$ must be a ramification point. Comparing the number of ends and the number of ramification points we get for $4$ ends that the set of ramification points and the set $\{p_1,...,p_4\}$ are equal. Thus, the condition of Montiel and Ros that the branching values of the ramification points lie in an equator is equivalent  that the normals at the ends lie in an equator. By their work  \cite[Corollary~16]{MontielRos} we know that this cannot be the case for a surface without logarithmic ends.\\
 Unfortunately, with this approach we cannot treat more than $4$ ends. The number or ramifications is in general bigger than the number of ends. On the other hand, our approach above seems to require more knowledge about the kernel of the skew-symmetric matrix $M$ in order to deal with more than $4$ ends. It would be very interesting to know whether there is a minimal sphere with $m$ embedded planar ends such that the normals at the ends lie in an equator of $\Sp^2$. We know from the work of Montiel and Ros that in such a situation there must be ramification points that are not ends and their values do not lie in that equator. \\
 The authors think that it is not likely that there is an immersed minimal sphere with embedded planar ends where the ends lie in an equator.\\
  \textbf{Conjecture:} The span of the normals of the ends of an immersed complete minimal surface with embedded planar ends in $\R^3$ is always three-dimensional.
 \item After finishing the proof of Lemma~\ref{lem:M} and Lemma~\ref{lem:M2} the authors learned that the method of studying the skew-symmetric matrix $M$ was already used in the preprint \cite[Proof of Theorem~15]{KusnerSchmitt}.
 \end{itemize}
 \end{remark}

\begin{proposition}\label{prop:symm}
Let $\Psi:\Sigma \to\R^3$  be a closed Willmore immersion such that $X:= \frac{\Psi}{|\Psi|^2}:\Sigma\setminus\{p_1,..., p_m\} \to \R^n$ is a complete, immersed minimal surface with $m$ embedded planar ends, $m = 2p$ for a $p\in \N$. Assume furthermore that $\Psi$ has an orientation reversing $2p$-fold rotational symmetry around an axis of symmetry going trough $0=\Psi(p_i)=\Psi(p_j)$. We assume further that the span of the normals at the ends is not $2$-dimensional (which is shown above for $\Sigma=\mathcal{S}^2$ and $m=4$). Under these conditions the surface $\Psi$ is stable under variations that preserve the given symmetry.
\end{proposition}

\begin{remark}
\begin{itemize}
 \item  There exist Willmore spheres with $m=2p$ for each even $p\geq 2$ which have $2p$-fold orientation reversing symmetry and the $m$-fold point on the axis, see \cite{Kusner, Francis95, Francis97}. The most famous one is the Morin surface where $p=2$.
 \item For the Morin surface, the statement of Proposition~\ref{prop:symm} was conjectured in \cite[Section~1]{Francis97}.
\end{itemize}

\end{remark}

\begin{proof}
%  We have seen in the proof of Theorem~\ref{thm:leqm} that the Li-Yau inequality leads to the fact that if a variation does not resolve the $m$-fold point then the variation does not decrease the Willmore energy. We show now that in the given setting the $m$-fold point is never resolved.\\
We have seen in the proof of Theorem~\ref{thm:allcodim} that substracting $\W$-Jacobi fields from a given variational direction does not change the second variation of $\Psi$ into that direction. And we know from the proof of Theorem~\ref{thm:leqm} that $\Psi$ is stable under variations satisfying $\frac{d}{dt}|_{t=0}\Psi_t (p_i) =\vec v(p_i) = 0$ for all $i=1,...,m$. We show now that in the given situation we can always find a Jacobi field with $\vec j(p_i)= \vec v(p_i)$ for all $i=1,...,m$ which we will then substract from the given variation in order to see that the second variation is always non-negative.\\
 Let $S\in SO(3)$ be the matrix correspondig to the rotation by $\frac{2\pi}{m}$ around an axis in $\R^3$ which we can assume to be spanned by $e_3$. In the given situation we have a diffeomorphism $s:\Sigma \to\Sigma$ such that
 \begin{align} \label{eq:3}
  S(\Psi(p)) = \Psi(s(p)) \ \ \forall p\in \Sigma\ \ \text{ and } s(p_i) = p_{i+1} 
 \end{align}
after possibly reordering the points $p_i$. We use the convention $p_{m+1}: = p_1$. 
%As the $m$-fold point is on the symmetry axis we get that $S(\Psi(p_i)) = \Psi(p_{i+1}) = \Psi(p_i)$ for all $i=1,...,m$.
 After each rotation by $\frac{2\pi}{m}$ the normal changes sign. This means that $n_{\Psi}(p_{i+1}) = - S n_{\Psi}(p_i)$, where $n_\Psi$ denotes a unit normal along $\Psi$.\\ 
 Now let $\vec v$ be a vector and $\Psi_t$ a smooth variation with $\frac{d}{dt}\big|_{t=0}\Psi_t = \vec v$ that preserves the given symmetry. Using (\ref{eq:3}) for $\Psi_t$, differentiating with respect to $t$ and projecting onto the normal part we get that 
 \begin{align*}
  S\vec v(p)^{\perp_{s(p)}} = \vec v (s(p))^{\perp_{s(p)}}.
 \end{align*}
 We use the notation $v(p) : =  \vec v(p) \cdot n_{\Psi}$ and compute
 \begin{align*}
  v(p_{i+1}) &=  \vec v(p_{i+1}) \cdot n_{\Psi}(p_{i+1})=   S\vec v(p_i) \cdot  n_{\Psi}(p_{i+1})\\
  &=  - \vec v (p_{i})\cdot n_{\Psi}(p_{i}) = - v(p_i).
 \end{align*}
Since the Willmore surface $\Psi$ is invariant under translations we know that $j(p): =  n_{\psi}(p)\cdot e_3$ is a $\W$-Jacobi field, i.e.\ $\delta^2\W(\Psi)(j n_{\Psi}, jn_{\Psi}) = 0$. By the same reasoning as above  and since $e_3$ is the axis of the symmetry we get that $-j(p_i) = j(p_{i+1})$ for all $i=1,...,m$. For an arbitrary symmetry preserving variational direction $\vec v = vn_{\Psi}$ we substract the Jacobi field $J(p) n_{\Psi}: =  \frac{v(p_1)}{j(p_1)} j (p) n_{\Psi}(p)$ from $\vec v$ and see that
\begin{align*}
 \delta^2 \W(\Psi)(\vec v, \vec v) = \delta^2\W(\Psi)(\vec v - J n_{\Psi}, \vec v - J n_{\psi})\geq 0 
\end{align*}
because $ v(p_i) - J(p_i)  = 0$ for all $i=1,...,m$ (as in the proof of Theorem~\ref{thm:allcodim}).\\
Note that we have used here that $j(p_1) \neq 0$. For $m=4$ and $\Sigma =\mathcal{S}^2$, this follows from the fact that the normals cannot lie in an equator (proof of theorem Theorem~\ref{thm:forspheres}). For $m>4$ and $\Sigma =\mathcal{S}^2$ or for $\Sigma \neq \mathcal{S}^2$ we needed to assume this property. We think that in general the normals do not lie in an equator. But we cannot show that at the moment.
\end{proof} 

\begin{corollary} \label{cor:Morin}
 The Morin surface is stable under variations that preserve the four-fold, orientation reversing rotational symmetry.
\end{corollary}
\begin{proof}
 We have shown in the proof of Theorem~\ref{thm:allcodim} that for all minimal spheres with four ends the span of the normals at the ends is three-dimensional. We can, of course, compute that for the Morin surface explicitly with the parametrization found in \cite{Kusner}. It turns out that the normals of the Morin surface span a regular tetrahedron.
\end{proof}

\section{Expansion of the second variation of the Gauss curvature} \label{section3}
Parts of the following estimates in this section are already contained in \cite{Michelat}, where at least $C^2$ regularity of the normal variation on the Willmore surface is assumed. To close our argument we need to consider vectorfields with less regularity towards the ends. Hence we decided to present all needed estimate in a concise form.\\[-0.2cm]

Let us fix local conformal coordinates in a neighbourhood $D(p_i)$ around an end $p_i$ in $\Sigma$ such that $z(p_i)=0$. We may scale the coordinates such that we have 
\[ X_z \, dz =\left( - \frac{a}{z^2} +  Y(z) \right) \, dz \]
with $a \in \C^3, a^2 =0 , \abs{a}^2=2$ and $Y(z)$ is holomorphic and bounded (since the end is embedded and not logarithmic).
In particular this implies that 
\begin{align}\begin{split}
\label{eq:expansion of X_z} \frac12 \abs{X_z}^2 &= \abs{z}^{-4}( 1 - \Re(Y(z)\cdot \bar{a} z^2)+ \frac12\abs{z}^4 \abs{Y(z)}^2)= \abs{z}^{-4}(1+b(z)) \\ \abs{X(z)}^2 &= \abs{z}^{-2}(1+ c(z)) \end{split} \end{align}
where $b,c$ are smooth functions satisfying 
 \begin{align*} \sup_{\abs{z}<\epsilon_0}\left( \abs{z^{-2}b(z)}+ \abs{z^{-1}Db(z)} + \abs{D^2b(z)}\right) &\le C\\
 \sup_{\abs{z}<\epsilon_0} \left(\abs{z^{-1}c(z)}+ \abs{Dc(z)}+\abs{D^2c(z)}\right)&\le C %+\abs{D^3c(z)} &\le C\,.
 \end{align*}
Furthermore, we set $D_\epsilon(p_i)$ to be the preimage of the $\epsilon$ ball around $0$ in these coordinates, i.e.\ $D_\epsilon(p_i) =z^{-1}(B_\epsilon), B_\epsilon \subset \R^2$. We note that the expansion for $\frac12 \abs{X_z}^2$ can be used to estimate $K$ at an end as follows. Since $z$ are conformal coordinates we have 
\begin{align*} K \,\frac12 \abs{X_z}^2 &= -\frac12 \Delta\ln\left(\frac12 \abs{X_z}^2\right)= 2 \Delta \ln(\abs{z}) - \frac12\Delta\ln(1+b(z))\\&=- \frac{\Delta b(z)}{2+2b(z)} +  \frac{ \abs{\nabla b(z)}^2}{2(1+b(z))^2}=O(z^2); \end{align*}
where we used that $\Delta b(z) = \frac12 \Delta(\abs{z}^4  \abs{Y(z)}^2) = O(z^2)$ since $Y(z)\cdot \bar{a} z^2$ is holomorphic and $\nabla b(z) = O(z)$. In summary we conclude that 
\begin{align} K(z) = O(z^6), 
\label{eq:expansion Gauss curvature} 
\end{align}
which was already shown in \cite[Section~4.3]{Michelat}. 
The second variation of the Gau\ss\ curvature, \cite[Lemma 3.2]{Michelat} for a $w \in C^{2}(\Sigma\setminus \bigcup_i D_\epsilon(p_i))$ is given by 
\begin{align*} \frac{d^2}{dt^2} K_{g_t} \, d\operatorname{vol}_{g_t} = d\left(\Delta_gw \star dw - \frac{1}{2} \star d\abs{dw}_g^2\right) + d\left(2 K_g \star dw \right)
\end{align*}
Applying Stokes theorem on the set $\Sigma \setminus \bigcup_{i=1}^m D_\epsilon(p_i)$ we are left with estimating
\begin{align*} &\sum_{i=1}^m \int_{\partial D_\epsilon(p_i)} \Delta_gw  \frac{\partial w}{\partial \nu} - \frac{1}{2} \frac{\partial}{\partial \nu} \left(\abs{dw}_g^2\right) + \int_{\partial D_\epsilon(p_i)} 2K \frac{\partial w}{\partial \nu}\\
&= \sum_{i=1}^m I_i(w,w) + II_i(w). \end{align*}

This is summarised in the following lemma. 
\begin{lemma}\label{lem.expansion at an end}
Let $p$ be a planar embedded end of $\Sigma$ and let $w$ have an expansion of the form
\begin{equation}\label{eq:expansion w}w(z)= v \abs{X(z)}^2 + {\Re(\frac{\alpha}{z})} + \beta \ln\abs{z} + u(z)=v \abs{X(z)}^2 + w_1 \,; \end{equation}
 in the conformal coordinates $z$, 
where $v,\beta \in \R$, $\alpha \in \C$ and $u \in C^{2,\alpha}(\overline{B}_\epsilon)$. 
Then we have that
\begin{align}\label{eq:expansion at the end}
	&\int_{\partial D_\epsilon(p)} \Delta_gw  \frac{\partial w}{\partial \nu} - \frac{1}{2} \frac{\partial}{\partial \nu} \left(\abs{dw}_g^2\right) + \int_{\partial D_\epsilon(p)} 2K \frac{\partial w}{\partial \nu}\\\nonumber &
	=2v^2 \int_{\partial D_{\epsilon}(p)} \frac{\partial \abs{X}^2}{\partial \nu} - 8\pi v \beta + O(\epsilon)\;;
\end{align}
where the constant in $O(\epsilon)$ only depends on $X, \alpha, \beta$ and $\norm{u}_{C^2(D_\epsilon)}$. 
	
\end{lemma}
\begin{proof}
Let us write as before $\int_{\partial D_\epsilon(p)} \Delta_gw  \frac{\partial w}{\partial \nu} - \frac{1}{2} \frac{\partial}{\partial \nu} \left(\abs{dw}_g^2\right) + \int_{\partial D_\epsilon(p)} 2K \frac{\partial w}{\partial \nu} = I(w,w)+II(w,w)$.

We set $w_0(z):=v \abs{X(z)}^2$ and write $w=w_0 + w_1$.
Let us start with $II(w)$. By the above expansion of $w$ and \eqref{eq:expansion of X_z} we see that $\abs{Dw(z)}\le \frac{C}{r^3}$, $|z|=:r$. This implies that 
\[ \abs{K(z)}\abs{Dw(z)} = O(z^3) \] 
and so $II(w)= O(\epsilon^4)$. 

Let us continue with $I$. %After rotating the conformal coordinates $z$ (i.e. multiplying with $e^{i\alpha}$ ) we may assume that $a\in \R$ and so $\Re(\frac{a}{z}) = a \frac{\cos(\theta)}{r}$.
Due to the bilinearity of $I$ we are left with estimating $I(w_0,w_0)$, $I(w_0, w_1)$, $I(w_1,w_1)$. If $v=0$ we only have to estimate the last term, which we will show is of order $O(\epsilon)$ and we can directly conclude the lemma. In the following we will assume w.l.o.g that $v=1$ i.e. $v,\alpha, \beta, u \to 1, \frac{\alpha}{v}, \frac{\beta}{v}, \frac{u}{v}$. 
Recall that $X$ is harmonic. Therefore, we have $\Delta_g \abs{X}^2=2\abs{\nabla X}^2_g =4$. As we use conformal coordinates we get that 
\begin{align*} \abs{dw_0}^2_g &= 4\langle X, \frac{X_1}{\frac{1}{\sqrt{2}} \abs{\nabla X}} \rangle^2 +  4\langle X, \frac{X_2}{\frac{1}{\sqrt{2}} \abs{\nabla X}} \rangle^2 = 4\abs{X^T}^2 \\&= 4\abs{X}^2 - 4\abs{X^\perp}^2= 4w_0 - 4\abs{X^\perp}^2 \end{align*}
where $^T$ is the projection onto the tangent space and $^\perp$ the projection onto the normal space.
A further consequence of the choice of conformal coordinates is that for any $C^1$-function $f$ we have that
\[ \int_{\partial D_\epsilon(p)} \frac{\partial f}{\partial \nu} = \int_{\partial B_\epsilon} \sqrt{\frac12 \abs{\nabla X}^2}\,\frac{\partial f}{\partial \nu}= \int_{\partial B_\epsilon} -\partial_r f\,.\]

Hence we get that
\begin{align} \begin{split} I(w_0,w_0) &= \int_{\partial D_{\epsilon}(p_i)}  2\frac{\partial \abs{X}^2}{\partial \nu} - 2 \frac{\partial}{\partial \nu} \abs{X^\perp}^2\\
&=  2 \int_{\partial D_{\epsilon}(p_)} \frac{\partial \abs{X}^2}{\partial \nu}-4 \int_{\partial D_{\epsilon}(p)}  X\cdot N \, X \cdot\frac{\partial N}{\partial \nu} \\
&= 2 \int_{\partial D_{\epsilon}(p_i)} \frac{\partial \abs{X}^2}{\partial \nu} + O(\epsilon^2) 
\end{split} \label{eq:Computation}
\end{align}
In the second to last line we used that $X\cdot N$ is bounded and denoting by $a$ the second fundamental form of $\Sigma$
\[\abs{ X\cdot \frac{\partial N}{\partial \nu}}\le \abs{X}\abs{a}=\abs{X}\sqrt{ - K}=O(z).\]
Now let us estimate
\begin{align*}
	 I(w_1,w_0)&=\int_{\partial D_\epsilon(p)} 4 \frac{\partial w_1}{\partial \nu} + \Delta_g w_1 \frac{\partial w_0}{\partial \nu}- \frac{\partial}{\partial \nu} \left( \langle dw_0, dw_1\rangle_g \right)\\&= \int_{\partial D_\epsilon(p)} 4 \frac{\partial w_1}{\partial \nu} + III\,; 
\end{align*}
where we used that $\Delta_gw_0 = 4$. We want to show that
\[
	\int_{\partial D_\epsilon(p)}\frac{\partial w_1}{\partial \nu}= -\beta 2\pi + O(\epsilon)\text{ and } 
	 III = O(\epsilon). 
\]
By \eqref{eq:expansion w} we have that
\[ - \partial_r w_1 = \frac{1}{r}\Re(\frac{\alpha}{z}) - \frac{\beta}{r} + O(1).\]
Since $\int_{\partial D_\epsilon} \Re( \frac{a}{z})  = 0 $ we conclude the first part. \\
Using \eqref{eq:expansion of X_z} we get that
\begin{align*} -\frac{1}{\frac12\abs{X_z}^2}\frac{\partial w_0}{\partial r} &= 2 r\, \frac{1+c}{1+b} - r\,\frac{rc_r}{1+b}=2r + rd\\
\frac{1}{\frac12\abs{X_z}^2}\frac{\partial w_0}{\partial \theta}&=r^2\frac{\partial_\theta c}{1+b} 
 \end{align*}
 where due to \eqref{eq:expansion of X_z} we have $\sup_{\abs{z}<\epsilon_0} \left(\abs{z^{-1}d(z)} + \abs{Dd(z)}+\abs{D^2d(z)}\right)\le C$. Hence we have
 \begin{align*}
 III&= \int_{\partial B_\epsilon}-\frac{1}{\frac12\abs{X_z}^2}\frac{\partial w_0}{\partial r} \Delta w_1 + \frac{\partial}{\partial r} \left(\frac{1}{\frac12\abs{X_z}^2} \partial_r w_0 \partial_r w_1 +\frac{1}{r^2\frac12\abs{X_z}^2}  \partial_\theta w_0 \partial_\theta w_1 \right) \\
  &=\int_{\partial B_\epsilon}(2+d)(\partial_r(r \partial_rw_1)) + \frac{2+d}{r}\partial_{\theta\theta}w_1- \partial_r( (2+d)r \partial_r w_1)+\partial_r(\frac{\partial_\theta c}{1+b}\partial_\theta w_1)	\\
 &=\int_{\partial B_\epsilon}\frac{2+d}{r}\partial_{\theta\theta}w_1+ \frac{\partial}{\partial r}\left(\frac{\partial_\theta c}{1+b}\right)\partial_\theta w_1 + \frac{\partial_\theta c}{1+b} \partial_{r\theta}w_1 - \partial_rd \,r \partial_r w_1
% -r\partial_r w_1\partial_r d+\partial_{r \theta}w_1 \frac{\partial_\theta c}{1+b} + \partial_\theta w_1 \partial_r(\frac{\partial_\theta c}{1+b})\\
% &= \frac{2+d}{r}\partial_{\theta\theta}w_1 - r\partial_rw_1 f_1- \partial_{r\theta}w_1 f_2-\partial_\theta w_1 f_3.
\end{align*}
Let us start with the last term. Since we have $-r \partial_r w_1 = \Re(\frac{\alpha}{z}) - \beta +O(z)$ and $\partial_rd= \partial_rd(0)+ O(r)$ we conclude $-\partial_rd\, r \partial_r w_1= \partial_rd(0)\Re(\frac{\alpha}{z}) + O(1)$. This implies that 
\[\int_{\partial B_\epsilon} -\partial_rd r \partial_r w_1= \int_{\partial B_\epsilon} \partial_rd(0)\,\Re(\frac{\alpha}{z}) + O(\epsilon) = O(\epsilon). \]
For the second of last term, we can proceed similarly: We have $r\partial_{\theta r} w_1 = \Re(\frac{i\alpha}{z})+O(z)$ and $\left(\frac{\frac{1}{r}\partial_\theta c}{1+b}\right)=\left(\frac{\frac{1}{r}\partial_\theta c}{1+b}\right)(0)+O(z)$. Thus, as  before we get that 
\[\int_{\partial B_\epsilon} \left(\frac{\frac{1}{r}\partial_\theta c}{1+b}\right) r \partial_{\theta r} w_1= \int_{\partial B_\epsilon}  \left(\frac{\frac{1}{r}\partial_\theta c}{1+b}\right)(0)\,\Re(\frac{i\alpha}{z}) + O(\epsilon) = O(\epsilon)\,.\]
It remains to estimate the first two terms. They are estimated in a similar fashion. We recall that if $f $ is a $C^3$ function then $\frac1r\partial_\theta f=   \left(\frac1r\partial_\theta f\right)(0)+ O(z)$ and $\frac1r\partial_{r\theta} f=   \left(\frac1r\partial_{r\theta} f\right)(0)+ O(z)$ thus, there is a constant $d_1$ such that 
\[ \frac1r\partial_\theta d = d_1 +O(z) \text{ and } \frac1r \frac{\partial}{\partial r}\left(\frac{\partial_\theta c}{1+b}\right) = O(1). \]
Using  $\partial_\theta w_1 =\Re(\frac{-i\alpha}{z})+O(r) $
we integrate by parts once in the first term and conclude 
\begin{align*}
&\int_{\partial B_\epsilon}	\frac{2+d}{r}\partial_{\theta\theta}w_1+ \frac{\partial}{\partial r}\left(\frac{\partial_\theta c}{1+b}\right)\partial_\theta w_1= \int_{\partial B_\epsilon}	\left(-\frac1r \partial_\theta d+ \frac{\partial}{\partial r}\left(\frac{\partial_\theta c}{1+b}\right)\right)\partial_{\theta}w_1\\
&=\int_{\partial B_\epsilon} \left(-d_1\,\Re(\frac{-i\alpha}{z}) + O(1)\right) = 0 + O(\epsilon). 
\end{align*}
This concludes the mixed term.

Let us now finally estimate $I(w_1,w_1)$. Note that 
\begin{align*}
\Delta_g w_1 \frac{\partial w_1}{\partial \nu} = - \frac{1}{\frac12 \abs{X_z}^2} \Delta w_1 \partial_r w_1 = -\frac{r^4}{1+b} \Delta u \,\partial_r w_1 = O(r^2)	
\end{align*}
since $\abs{\partial_r w_1} \le\frac{C}{r^2}$. Using the fact that for a real valued function $f$ one has $\frac{1}{4}\abs{Df}^2= \abs{\partial_zf}^2$, where $\partial_z = \frac12 (\partial_x  - i \partial_y )$,  we have that
\begin{align*} \frac14 \abs{d w_1}_g^2 &= \frac{1}{\frac12\abs{X_z}^2}\abs{\partial_zw_1}^2 = \frac{r^4}{1+b}\left|-\frac{\alpha}{z^2} + \frac{\beta}{z} + \partial_zu\right|^2\\
&=\frac{1}{1+b}\abs{-\alpha+\beta z + z^2 \partial_zu}^2\,. \end{align*} 
This is a $C^1$ function. We thus deduce that $\frac{\partial}{\partial \nu} \abs{dw_1}_g^2$ is bounded on $D_\epsilon$ and we know that $I(w_1,w_1)=O(\epsilon)$. 
\end{proof}

\begin{remark}
 We note that we also re-proved the formula for the second variation (4.31) first proven by Michelat in \cite{Michelat} for a smooth variation $\vec v$. If the variation $\vec v = \psi n_\Psi$ is smooth, the coefficient $\beta$ is zero in the expansion of $w$ and $\psi(p) = v$. Thus, it remains to show that in (\ref{eq:expansion at the end}) the term $2\int_{\partial D_\epsilon (p)}\frac{|\partial X|^2}{\partial \nu}$ agrees with $+4\pi \frac{\operatorname{Res}_{p}(\vec \Phi, U)}{\epsilon^2}$ using the notation $\operatorname{Res}_{p}(\vec \Phi, U)$ from Definition-Proposition~4.6 in \cite{Michelat}. Note that we have chosen our parametrization in such a way that $\operatorname{Res}_{p}(\vec \Phi, U) = 2$ (see the beginning of Section~\ref{section3}).\\
 In order to show $2\int_{\partial D_\epsilon (p)}\frac{|\partial X|^2}{\partial \nu} = \frac{8\pi}{\epsilon^2} + O(\epsilon) $ we go back to the computation (\ref{eq:Computation}). We use the expansion  of $b$ and $c$ (in particular $c(0)=0$) at $z=re^{i\theta}$ and get that
\begin{align*}
&\sqrt{\frac12 \abs{\nabla X}^2}\;\frac{\partial \abs{X}^2}{\partial \nu}= - \frac{\partial \abs{X}^2}{\partial r}=\frac{2}{r^3} + \frac{2c(z)}{r^3}-\frac{Dc(z)\cdot z}{r^3}\\
&=\frac{2}{r^3} + \frac{2Dc(0)\cdot z + D^2c(0)(z,z) + O(z^3)}{r^3} - \frac{Dc(0)\cdot z + D^2c(0)(z,z)+O(z^3)}{r^3}\\
&=\frac{2}{r^3} + \frac{ Dc(0)\cdot z}{r^3} + O(1).   	
\end{align*}
And since 
\[ \int_{\partial B_\epsilon}\frac{Dc(0)\cdot z}{\epsilon^3} =0\, \] 
we get the form $I(w_0,w_0) = \frac{8\pi}{\epsilon^2} + O(\epsilon)$. All other terms remain the same as in the proof of Lemma~\ref{lem.expansion at an end}, therefore we have shown $2\int_{\partial D_\epsilon (p)}\frac{|\partial X|^2}{\partial \nu} = \frac{8\pi}{\epsilon^2} + O(\epsilon) $ and thus
\begin{align*}
-\int_{\partial D_\epsilon(p)} \Delta_gw  \frac{\partial w}{\partial \nu} - \frac{1}{2} \frac{\partial}{\partial \nu} \left(\abs{dw}_g^2\right) + \int_{\partial D_\epsilon(p)} 2K \frac{\partial w}{\partial \nu}
	= - v^2 \frac{8\pi}{\epsilon^2} + O(\epsilon)
	\end{align*}
	if $\psi$ is smooth.
\end{remark}

At this point let us shortly fix the following property of functions with an expansion as in Lemma~\ref{lem.expansion at an end}.
\begin{lemma}\label{lem.sobolev space}
	Under the assumptions of Lemma~\ref{lem.expansion at an end} we have that $\frac{w}{\abs{X(z)}^2}$ is an element of $W^{2,p}(D_{\epsilon}(p_i), d\mu_{\hat{g}})\cap C^1(D_{\epsilon}(p_i))$ for all $p<\infty$ with respect to the metric $\hat{g}$ on the Willmore surface. 
\end{lemma}

\begin{proof}
By \eqref{eq:expansion w} and \eqref{eq:expansion of X_z} we have in the conformal coordinates $z$ 
\begin{align*}\label{eq:expansion w/|X|^2}
\frac{w}{\abs{X(z)}^2}&= v + \Re\left(\frac{\alpha}{z\abs{X(z)}^2}\right) + \beta\frac{\ln\abs{z}}{\abs{X(z)}^2} + \frac{u(z)}{\abs{X(z)}^2}\\
&=v +  \Re\left(\frac{\alpha \bar{z}}{1+c(z)}\right)+ \beta \frac{\abs{z}^2\ln\abs{z}}{1+c(z)} + \frac{\abs{z}^2u(z)}{1+c(z)}\\
&=\tilde{u}(z)+ \beta \frac{\abs{z}^2\ln\abs{z}}{1+c(z)}\,;
\end{align*}
where $\tilde{u}$ is a smooth function. Thus $\tilde{u}$ is clearly in element $W^{2,2}(D_{\epsilon}(p_i))\cap C^1(D_{\epsilon}(p_i))$. 
We write the second term as $\frac{f(\abs{z})}{1+c(z)}$ with $f(t)=t^2\ln(t)$. We have $f'(t)= t(2\ln(t) +1)$ is a bounded continuous function in $0$ and $f''(t)= 2\ln(t)+3$ is unbounded.
Therefore we conclude that the first derivative is bounded and continuous in $0$ since
\begin{align*}
D\frac{f(\abs{z})}{1+c(z)}= \frac{f'(\abs{z})}{1+c(z)} \frac{z}{\abs{z}}-\frac{f(\abs{z})}{(1+c(z))^2} Dc(z).
\end{align*}
For the second derivative we have 
\[D^2\frac{f(\abs{z})}{1+c(z)}= \frac{f'(\abs{z})}{1+c(z)} D^2\abs{z}+\frac{f''(\abs{z})}{1+c(z)} D\abs{z}\otimes D\abs{z} + R(z)\]
where $R(z)$ contains at most one derivative of $f(\abs{z})$ and is therefore continuous in $0$. The eigenvalues of the first part are $\frac{f''(\abs{z})}{1+c(z)}$ and $\frac{f'(\abs{z})}{\abs{z}(1+c(z))}$ corresponding to the eigenvectors $\frac{z}{\abs{z}}$ and $\frac{z^\perp}{\abs{z^\perp}}$. Both are unbounded in $0$ so the function is not $C^2$. But they are bounded by $\frac{2(\abs{\ln(\abs{z}}}{\abs{1+c(z)}}\le 4\abs{\ln\abs{z}}$. Since $\ln\abs{z}\in L^p(B_R)$ for all $p<\infty$ we conclude the argument. 

\end{proof}

\begin{theorem}\label{thm:2nd var well-def new}
 Let $\Psi:\Sigma \to\R^3$ be a closed Willmore immersion such that $X:= \frac{\Psi}{|\Psi|^2}:\Sigma\setminus\{p_1,..., p_m\} \to \R^3$ is a complete, immersed minimal surface with $m$ embedded planar ends, $m\geq 1$. 
 Take a function $w\in W^{2,2}(\Sigma, d\mu_{\hat g}) $ with the additional assumption that $w$ has around an end $p_i$ of $X$ an expansion of the form
 \begin{align} \label{expansion new}
w(z)= v_i \abs{X(z)}^2 + \Re(\frac{\alpha_i}{z}) + \beta_i \ln\abs{z} + u_i(z) \end{align}
in the conformal coordinates $z$ from the beginning of this section, where $v_i,\beta_i \in \R$, $\alpha_i \in \C$ and $u_i \in C^{2,\alpha}(\overline{B}_\epsilon)$. \\
We notice that the function $\psi:= \frac{w}{|X|^2}$ satisfies $\psi\in W^{2,2}(\Sigma, d\mu_{\hat g})$ by Lemma~\ref{lem.sobolev space}.\\ 
Then the second variation of $\W(\Psi)$ into the direction $\psi n_\Psi$ is well-defined. In formulas, the second variation reads as follows: For every small $\epsilon>0$ the second variation is
\begin{align*}
 \delta^2 \W(\Psi)(\psi,\psi) &=   \int_{\Sigma\setminus \cup_{i=1}^m D_\epsilon (p_i)} \tfrac{1}{2} \left(L w\right)^2 d \mu_g  \\
  & \qquad \ \ -  2\sum_{i=1}^m v_i^2\int_{\partial D_{\epsilon}(p_i)} \frac{\partial \abs{X}^2}{\partial \nu} + 8\pi  \sum_{i=1}^m v_i \beta_i + R_\epsilon,
\end{align*}
where $g = X^\# \delta$ is the pullback metric of the minimal surface and $D_\epsilon (p_i)= z^{-1}(B_\epsilon)$ is the preimage of the Euclidean ball $B_\epsilon\subset\R^2$ under the chart $z$. The term $R_\epsilon$ is an error term with the property $R_\epsilon \to 0$ for $\epsilon \to 0$. Furthermore, $R_\epsilon$ only depends on the sum of the $v_i,  \alpha_i, \beta_i$ and $\|u_i\|_{C^{2}(\bar D_\epsilon (p_i))}$.
\end{theorem}

\begin{proof}
We choose $w_k\in C^{2,\alpha}(\Sigma\setminus\{p_1,...,p_m\})$ with the following properties: $w_k=w$ on $D_{\epsilon}(p_i)$ for all $i=1,...,m$ and $w_k\to w$ in $W^{2,2}(\Sigma\setminus \cup_{i=1}^m D_{\frac{\epsilon}{2}}(p_i), d\mu_{\hat g})$. \\
Furthermore, for $2\delta <\frac{\epsilon}{2}$ we choose a cut-off function
\begin{align*}
\eta_\delta=
 \begin{cases}
  0 \text{ on } \cup_{i=1}^m D_\delta(p_i)\\
  1 \text{ on } \Sigma \setminus\cup_{i=1}^m D_{2\delta}(p_i)\\
  \text{ smooth interpolation otherwise}
 \end{cases},
 \ \ 0 \leq \eta_\delta\leq 1, \ \ \eta_\delta \in C^\infty(\Sigma).
\end{align*}
The function $w_k$ coincides with $w$ on $D_\epsilon (p_i)$. Thus, it has the same expansion (\ref{expansion new}) around each end. Define
\begin{align*}
 w_k^\delta : =
 \begin{cases}
 |X|^2 v_i + \Re(\frac{\alpha_i}{z}) + \beta_i\eta_\delta \ln |z| + u_i(z) \text{ on } \bar D_{\frac{\epsilon}{2}} (p_i) \\
 w_k \text{ on } \Sigma \setminus  \cup_{i=1}^m D_{\frac{\epsilon}{2}} (p_i).
 \end{cases}
\end{align*}
By this construction we approximated only the logarithmically growing part of $w$ around each end. We define
\begin{align*}
 \psi_k^\delta : = \frac{w_k^\delta}{|X|^2}.
\end{align*}
We notice that $\psi_k^\delta$ is a $C^{2,\alpha}$-function on all of $\Sigma$, and $\psi_k^\delta \to \psi_k = \frac{w_k}{|X|^2}$ as $\delta \to 0$ in $W^{2,2}(\Sigma, d\mu_{\hat g})$ and $\psi_k \to \psi$ as $k\to \infty$ in $W^{2,2}(\Sigma, d\mu_{\hat g})$.\\
Consider the family of $C^{2,\alpha}$-immersions $\Psi_t = \Psi + t\psi_k^\delta n_{\Psi}$, $t\in (-\epsilon_0, \epsilon_0)$. We define the tracefree second fundamental form $\mathring{A}_{ij} = A_{ij} - \frac{1}{2}\vec H \hat g_{ij}$. The quantity $|\mathring{A}|^2_{\hat g} d\mu_{\hat g}$ with $\hat g = \Psi^\# \delta$ is pointwise conformally invariant, see for example \cite[Section~1.2]{KuwertSch}. 
Because of the Gau\ss-Bonnet Theorem and the Gau\ss\ equation $\frac{1}{4}|\vec H|^2 = K + \frac{1}{2}|\mathring{A}|^2$ the Willmore functional $\W$ and the functional $\mathcal{T}(\Psi) := \frac{1}{2}\int_\Sigma |\mathring{A}|^2 d\mu_{\hat g} $ differ only by an additive constant. Thus, we have that $ \frac{d^2}{dt^2} \W(\Psi_t)\big|_{t=0}= \frac{d^2}{dt^2} \mathcal{T}(\Psi_t)\big|_{t=0} $. We mostly suppress the dependence of $t$ in the notations. 
We compute 
\begin{align} \begin{split}
 \delta^2 & \W(\Psi)(\psi_k^\delta, \psi_k^\delta  ) =\frac{d^2}{dt^2} \mathcal{T}(\Psi_t)\big|_{t=0}  \\
 &=\frac{d^2}{dt^2} \mathcal{T}(\frac{\Psi_t}{|\Psi_t|^2} \big|_{\Sigma\setminus \cup_{i=1}^m D_\epsilon (p_i)})\big|_{t=0} + \frac{d^2}{dt^2} \big|_{t=0} \int_{\cup_{i=1}^m D_\epsilon (p_i)}  \tfrac{1}{2} |\mathring{A}|^2\, d\mu_{\hat g} \\
 &= \frac{d^2}{dt^2} \big|_{t=0}  \big\{\int_{\Sigma\setminus \cup_{i=1}^m D_\epsilon (p_i)}  \tfrac{1}{4} |\vec H|^2 - K_g\, d\mu_{ g} + \int_{\cup_{i=1}^m D_\epsilon (p_i)}  \tfrac{1}{2} |\mathring{A}|^2\, d\mu_{\hat g} \big\}\\
 & =  \int_{\Sigma\setminus \cup_{i=1}^m D_\epsilon (p_i)} \frac{d^2}{dt^2} \big|_{t=0}\left( \tfrac{1}{4} |\vec H|^2 - K_g \right)d\mu_{ g} + \int_{\cup_{i=1}^m D_\epsilon (p_i)} \frac{d^2}{dt^2} \big|_{t=0} \tfrac{1}{2} |\mathring{A}|^2\, d\mu_{\hat g}\\
 & = \int_{\Sigma\setminus \cup_{i=1}^m D_\epsilon (p_i)} \tfrac{1}{2}\left( L w_k\right)^2 - d\left(\Delta_gw_k \star dw_k - \frac{1}{2} \star d\abs{dw_k}_g^2\right) - d\left(2 K_g \star dw_k \right) d\mu_g\\
 & \qquad \ \ + \int_{\cup_{i=1}^m D_\epsilon (p_i)} \frac{d^2}{dt^2} \big|_{t=0} \tfrac{1}{2} |\mathring{A}|^2\, d\mu_{\hat g} \end{split} \label{eq:vor limit new}
\end{align}
where we used several things in this computation:
\begin{itemize}
 \item $\psi_k^\delta$ is of regularity $C^{2,\alpha}$. Only for $C^2$ variations we a priori know 
 \begin{align*}
  \delta^2 \W(\Psi)(\psi_k^\delta, \psi_k^\delta) = \frac{d^2}{dt^2} \W(\Psi_t)\big|_{t=0}.
 \end{align*}
 \item We used the conformal invariance of $|\mathring{A}|^2 d\mu_{\hat g}$ on $\Sigma\setminus \cup_{i=1}^m D_\epsilon (p_i)$.
 \item We also use that $\psi_k^\delta$ is smooth for interchanging differentiation and integration. In the first integral we have changed the metric conformally to $g$ which causes trouble at the ends. But we restrict ourselves to $\Sigma\setminus \cup_{i=1}^m D_\epsilon (p_i)$ which is a compact part of the minimal surface $X$. Therefore, differentiation and integration can be interchanged for the smooth variation $\Psi_t$. On the other part $ \cup_{i=1}^m D_\epsilon (p_i)$ we stay on the compact surface with the smooth metric, therefore interchanging is not problematic.
 \item In the last step we used the formulas for the second derivatives of $ |\vec H|^2 d\mu_g$ and $ K_g d\mu_g$ for the a compact piece of the minimal surface $X$ computed in \cite[p.15 and Lemma~3.2]{Michelat}. For that, we also needed $\frac{d}{dt}\frac{\Psi_t}{|\Psi_t|^2}\big|_{t=0} = |X|^2 \psi_k^\delta n_{X}$ from \cite[(4.6)]{Michelat}. %TODO cite the right numbers in the published file
 And we did not write the $\delta$ because we are on the piece of the surface where $w_k^\delta= w_k$.
\end{itemize}
On the set $\cup_{i=1}^m D_\epsilon(p_i) \subset \Sigma$ we notice that
\begin{align}\label{ineq:on epsilon new}
  \big|\int_{\cup_{i=1}^m D_\epsilon(p_i)} \frac{d^2}{d t^2}\big|_{t=0} |\mathring{A}|^2 d\mu_{\hat g} \big| \leq C \|\psi_k^\delta\|_{W^{2,2}(\cup_{i=1}^m D_\epsilon(p_i))}
\end{align}
for a constant $C$ not depending on $\epsilon$, $k$ nor $\delta$. Thus, by letting at first $\delta \to 0$ and then $k\to \infty$ this term is of the order $O(\| \psi\|_{W^{2,2}(\cup_{i=1}^m D_\epsilon (p_i))})$. Taking the computations of Lemma~\ref{lem.sobolev space} into account we see that $\| \psi\|_{W^{2,2}(\cup_{i=1}^m D_\epsilon (p_i))}$ only depends on the sum of the $v_i,  \alpha_i, \beta_i$ and $\|u_i\|_{C^{2}(\bar D_\epsilon (p_i))}$. And obviously, $\| \psi\|_{W^{2,2}(\cup_{i=1}^m D_\epsilon (p_i))} \to 0$ when $\epsilon \to 0$. \\
It remains to simplify the terms in (\ref{eq:vor limit new}) that come from the variation of $K_g$ and are boundary terms. We use Lemma~\ref{lem.expansion at an end} and Stokes' theorem to see that 
\begin{align}\label{equ:on complement new} \begin{split}
 \int_{\Sigma\setminus \cup_{i=1}^m D_\epsilon (p_i)} &   d\left(\Delta_gw_k \star dw_k - \frac{1}{2} \star d\abs{dw_k}_g^2\right) + d\left(2 K_g \star dw_k \right) d\mu_g = \\
  & \qquad \qquad \qquad  2\sum_{i=1}^m v_i^2 \int_{\partial D_\epsilon(p_i)} \frac{\partial|X|^2}{\partial \nu} - 8\pi \sum_{i=1}^m v_i \beta_i + \tilde u \end{split}
\end{align}
with $\tilde u \in O(\epsilon)$  only depending on the sum of the $v_i,  \alpha_i, \beta_i$ and $\|u_i\|_{C^{2}(\bar D_\epsilon (p_i))}$.\\
We finally use that the Index form $\delta^2\W(\Psi)(\cdot,\cdot)$  is continuous with respect to the $W^{2,2}$-norm on $(\Sigma, d\mu_{\hat g})$ (using that $\Sigma$ is compact). We use (\ref{eq:vor limit new}), (\ref{ineq:on epsilon new}) and (\ref{equ:on complement new}) to get that
\begin{align*}
 \delta^2 \W (\Psi)(\psi,\psi) & = \int_{\Sigma\setminus \cup_{i=1}^m D_\epsilon (p_i)} \tfrac{1}{2}\left( L w \right)^2 d\mu_g -  2\sum_{i=1}^m v_i^2 \int_{\partial D_\epsilon(p_i)} \frac{\partial|X|^2}{\partial \nu} \\
 & \qquad\qquad\qquad+ 8\pi \sum_{i=1}^m v_i \beta_i + R_\epsilon,
\end{align*}
where we also used $\int_{\Sigma\setminus \cup_{i=1}^m D_\epsilon (p_i)} \tfrac{1}{2}\left( L w_k \right)^2d \mu_g \to \int_{\Sigma\setminus \cup_{i=1}^m D_\epsilon (p_i)} \tfrac{1}{2}\left( L w \right)^2$ for $k \to\infty$. The term $R_\epsilon$ contains $\tilde u\in O(\epsilon)$ and the quantity coming from (\ref{ineq:on epsilon new}).\\
As $\epsilon>0$ is fixed, all remaining terms are finite. Thus, the second variation of $\W$ is well-defined for functions $\psi$ as given in the theorem.

\end{proof}

\begin{theorem}\label{thm:second variation formula}
Let $\Psi:\Sigma \to\R^3$ be a closed Willmore immersion such that $X:= \frac{\Psi}{|\Psi|^2}:\Sigma\setminus\{p_1,..., p_m\} \to \R^3$ is a complete, immersed minimal surface with $m$ embedded planar ends, $m\geq 1$. Consider a function $w$ as in Theorem~\ref{thm:2nd var well-def new} but additionally having the same leading value at each end, i.e.\ $w$ has around an end $p_i$ of $X$ an expansion of the form
\begin{align*}w(z)= v \abs{X(z)}^2 + \Re(\frac{\alpha_i}{z}) + \beta_i \ln\abs{z} + u_i(z) \end{align*}
in conformal coordinates $z$ as before, where $v,\beta_i \in \R$, $\alpha_i \in \C$ and $u_i \in C^{2,\alpha}(\overline{B}_\epsilon)$. We use the notation
\begin{align*}
w = : v \abs{X(z)}^2 + w_1. 
\end{align*}
Then the second variation of $\W(\Psi)$ into direction $\psi n_{\psi}$ with $\psi = \frac{w}{|X|^2}$ can be computed by the formula
\begin{align}\label{eq:second variation}
\begin{split}
 \delta^2 \W(\Psi)(\psi,\psi) &=  \frac{1}{2} \int_{\Sigma} \left(L w_1\right)^2 d \mu_g \\
 & + 2 v  \int_{\Sigma} L w_1 (2 - K_g |X|^2) d\mu_g + 8\pi v\sum_{i=1}^m \beta_i\\
 & + 2 v^2\int_{\Sigma} \left( K_g^2 |X|^4 - 4 K_g |X|^2 \right) d\mu_g 
 \end{split}
\end{align}

\end{theorem}
\begin{proof}
 We use the form $w = v|X|^2 + w_1$ and compute with the formula from Theorem~\ref{thm:2nd var well-def new}, Lemma~\ref{lem.expansion at an end} and $\Delta_g|X|^2 = 4$
 \begin{align*}
  \delta^2 \W(\Psi)(\psi,\psi) &= \lim_{\epsilon \to 0} \Big\{ \frac{1}{2} \int_{\Sigma\setminus \cup_{i=1}^m D_\epsilon (p_i)} \left( 4 v - 2 K_g |X|^2 v + L w_1\right)^2 \\
   & \qquad \ \ - 2v^2 \sum_{i=1}^m\int_{\partial D_{\epsilon}(p_i)} \frac{\partial \abs{X}^2}{\partial \nu} + 8\pi v \sum_{i=1}^m \beta_i \Big\}\\
   & = \lim_{\epsilon \to 0} \Big\{  \int_{\Sigma\setminus \cup_{i=1}^m D_\epsilon (p_i)}  8 v^2 +  4 v \left(- 2 K_g |X|^2 v + L w_1\right) d\mu_g \\
  & \qquad \ \  +\int_{\Sigma\setminus \cup_{i=1}^m D_\epsilon (p_i)}  \frac{1}{2}\left(- 2 K_g |X|^2 v + L w_1\right)^2 d\mu_g \\
   & \qquad \ \ - 2v^2 \int_{\Sigma\setminus \cup_{i=1}^m D_\epsilon (p_i)} \Delta_g \abs{X}^2 d\mu_g + 8\pi v \sum_{i=1}^m \beta_i \Big\},
 \end{align*}
where we also used Stokes' theorem to bring it back to $\Sigma\setminus \cup_{i=1}^m D_\epsilon (p_i)$. Using $\Delta_g|X|^2 = 4$ again we see that the integral $8 v^2 \int d\mu_g$ cancels out. The terms left over are precisely the terms in (\ref{eq:second variation}) but still on $\Sigma\setminus\cup_{i=1}^m D_\epsilon (p_i)$. In a last step we take the limit $\epsilon \to 0$ and integrate over all of $\Sigma$.
\end{proof}

\begin{theorem}\label{thm:logJac implies index}
Let $\Psi:\Sigma \to\R^3$ be a closed Willmore immersion such that $X:= \frac{\Psi}{|\Psi|^2}:\Sigma\setminus\{p_1,..., p_m\} \to \R^3$ is a complete, immersed minimal surface with $m$ embedded planar ends, $m\geq 1$. \\
Assume that there exists a logarithmically growing Jacobi field on $X$, i.e.\ a function $ u\in C^{2,\alpha}(\Sigma\setminus\{p_1,...,p_m\}) $, $L u = 0$, with expansion
$$u = \beta_i \log |z| + \tilde u_i (z) \ \text{ with } \ \tilde u_i \in C^{2,\alpha}(\overline{B_\epsilon}), \ \beta_i \in \R,  \  \sum_{i=1}^m \beta_i \not = 0.$$
at each end $p_i$ in the local conformal coordinates $z$ from the beginning of this section. \\
Then we have that $\Ind(\Psi) \geq 1$. 
\end{theorem}

\begin{proof}
 For every such Jacobi field $u$ we define $w:= v|X|^2 + u$ for $v\in \R$ to be chosen later. As $L u=0$ we get from formula (\ref{eq:second variation}) that
 \begin{align*}
  \delta^2\W(\Psi)(\frac{w}{|X|^2}, \frac{w}{|X|^2}) =  8\pi v\sum_{i=1}^m \beta_i  + 2 v^2\int_{\Sigma} \left( K_g^2 |X|^4 - 4 K_g |X|^2 \right) d\mu_g.
 \end{align*}
The last integral is a positive constant $c_X^2$ only depending on the geometry of $X$. Hence, using the notation $\beta: = \sum_{i=1}^m \beta_i$ we get that
 \begin{align*}
  \delta^2\W(\Psi)(\frac{w}{|X|^2}, \frac{w}{|X|^2}) =  v(8\pi  \beta + 2 v c_X^2).
 \end{align*}
 For $v= -\frac{2\pi \beta}{c_X^2}$ the expression on the right hand side is $-\frac{8\pi^2 \beta^2}{c_X^2}<0$. We emphasize that Theorem~\ref{thm:2nd var well-def new} tells us that the $\delta^2 \W(\Psi)(\frac{w}{|X|^2}, \frac{w}{|X|^2})$ is well-defined.
\end{proof}

\begin{remark}
 The Morin surface \cite{Morin78, Kusner} is a Willmore sphere in $\R^3$ with a quadruple point and $16\pi$ energy. By Theorem~\ref{thm:forspheres} it has Morse Index less or equal than one. We will know explain why it actually has Morse Index equal to one. \\
 We will prove the precise Index for all (smooth) Willmore spheres in the following section. 
%  this property for all Willmore spheres which come from a minimal sphere with four embedded planar ends in the following section. 
  But we want to note here that the above Theorem~\ref{thm:logJac implies index} can also be used to show that the Morin surface has Index one. This can be done as follows: \\
 The work of Montiel and Ros \cite[Theorem~25]{MontielRos} implies that the nullity, $\dim K$ of the inverted Morin surface is five. By the work of Perez and Ros \cite[Lemma~5.2]{PerezRos} it follows that the space of Jacobi fields that have at most logarithmic growth at the ends is $m + \dim K_0$, where $m=\# \text{ends}$ and $K_0$ is the space of bounded Jacobi fields $u$ with $u(p_i)=0$ for all $i$. Hence the number of logarithmic growing Jacobi fields is $m+\dim K_0 - \dim K$. In particular one needs to show that $\dim K_0=2$ for the Morin surface.
 
  The three coordinates of $n_\Psi$ are bounded Jacobi fields, but they satisfy $u(p_i)\neq 0$, which implies that $\dim K_0 \le 2$. There is one other Jacobi field known, namely the support function $u= X\cdot n_X$, see \cite{MontielRos}. In order to prove $u(p)=0$ for this Jacobi field we can compute that the assymptotic planes to the four ends all meet at one point (w.l.o.g.\ the origin). This implies that the support function $u$ is a tangential variation at the ends, which implies $u \in K_0$.\\
  The last missing Jacobi field of $X$ is the support function of the conjugated minimal surface corresponding to $X$. Again, by computing that the assymptotic planes meet at one point we get that this Jacobi field is also in $K_0$. As it is independent of the support function of $X$, we get that $\dim K_0 = 2$.\\
It remains to compute $\sum_{i=1}^4 \beta_i \neq 0$. By a computation that we will do in detail in Proposition~\ref{prop:space of logarithmic growing Jacobi fields on spheres} we get that $\beta_i=1$ for all $i=1,...,4$ (using $\sum_{i=1}^4 n_X(p_i) =0$ for the Morin surface). Thus, Theorem~\ref{thm:logJac implies index} gives a positive Index.
\end{remark}

\section{Jacobi fields with logarithmic growth} \label{section4}

In this section we want to proof the existence of Jacobi fields with logarithmic growth on spheres with a least four ends. We are going to combine ideas of \cite{PerezRos} and \cite{MontielRos} to establish their existence. Our general set up follows closely \cite[section 5]{PerezRos}. 

As in the previous section we fix local conformal coordinates in a neighbourhood $D(p_i)$ around each of the ends $p_i$ of the minimal immersion $X: \Sigma \to \R^3$ such that $z(p_i)=0$. 

The Jacobi operator on $\Sigma$ is given by $L=\Delta_g + \abs{\nabla n_X }_g^2$, where $n_X$ is the Gauss map of the minimal immersion $X$ and $g=X^\sharp\delta$. We note that $L$ is conformally invariant, hence we may ``compactify'' $L$ by multiplying the metric $g$ by a stirctly positive conformal factor $\lambda$. We choose $\lambda$ in such a way that in each $D(p_i)$ we have %we have in the fixed conformal coordinates $z$ domain $D(p_i)$ an expansion 
we have $\lambda(z)=\abs{z}^4$. We set $\hat{g}=\lambda g$ and $\hat{L}=\lambda L = \Delta_{\hat{g}} + \abs{\nabla n_X}^2_{\hat{g}}$. For our purpose it is convenient to choose $\lambda$ such that $\hat{g} = \Psi^\sharp \delta$, because for this choice  $\hat{L}$ agrees with the Jacobi operator of the Willmore immersion $\Psi:\Sigma \to \R^3$. 

Let $E=E(\Sigma) \subset C^{2,\alpha}(\Sigma \setminus\{p_1, \dotsc, , p_m\})$ be the space of functions $u$ that have in a neighbourhood $p_i$ inside the domains $D(p_i)$ the following expansion
\begin{equation}\label{eq:expansion lin+log}
	u(z)= \Re(\frac{a_i}{z}) + \alpha_i \ln\abs{z} + \mathsf{u}(z)
\end{equation}
 with $a_i \in \C$, $\alpha_i \in \R$ and $\mathsf{u}\in C^{2,\alpha}(D_\epsilon)$.
 To get a better description of $E_1$ let us fix for each $p_i$ functions on $D_\epsilon(p_i)=\{ x \in \Sigma \colon \abs{z}<\epsilon \} \subset D(p_i)$
 \begin{align}\label{eq:log functions}
 	l_i(z)&:= \eta(\abs{z}) \ln(\abs{z})\\ \nonumber
 	s_i(z)&:=\eta(\abs{z}) \frac{1}{z}
 \end{align}
where $\eta$ is a smooth non-negative cut-off function with $\eta=1$ for $\abs{z}\le\frac\epsilon2$ and $\eta=0$ for $\abs{z}\ge \epsilon$. 
Consider the $m$-dimensional vector space $V=V(\Sigma)=\{ \sum_{i=1}^m \alpha_i l_i \colon  \alpha=(\alpha_1, \cdots, \alpha_m) \in \R^m\}$ and the (real) $2m$-dimensional vector space $V_1=V_1(\Sigma)=\{ \sum_{i=1}^m \Re(a_i s_i(z)) \colon  a=(a_1, \cdots, a_m) \in \C^m\}$. If we denote by $C^{l,\alpha}(\Sigma)$ the set of functions that are $l$-times differentiable and the $l$th derivative is $\alpha$-H\"older continuous on the compact manifold $\Sigma$, we have 
\[ E=C^{2,\alpha}(\Sigma)\oplus V \oplus V_1. \]
Furthermore, we consider the subspace $E_1\subset E$ of all functions in $E$ that do not have a logarithmic part i.e.\ $\alpha =0$ or
\[ E_1 = C^{2,\alpha}(\Sigma) \oplus V_1. \]
We remark that the space $E_0:=C^{2,\alpha}(\Sigma) \oplus V$ had been investigated in \cite{PerezRos}.
The vector spaces $E, E_0,E_1$ are Banach space and in fact the topology is independent of the specific choice of the functions $l_i, s_i$. Since $\frac{1}{z}$ and $\ln(\abs{z})$ are harmonic and each $p_i$ is a branch point of $n_X$, compare \eqref{eq:expansion Gauss curvature}, $\hat{L}: E \to C^{0,\alpha}(\Sigma)$ is a bounded linear operator. The same is valid for $E_0,E_1$ in place of $E$.\\

A helpful tool will be the following integration by parts formula, which is a generalization of \cite[lemma 5.1]{PerezRos}.

\begin{lemma}\label{lem.integration by parts formula}
Let $p$ be a planar embedded end of $\Sigma$ and let $u,v$ have in the conformal coordinates $z$ at the end $p$ an expansion of the form \eqref{eq:expansion lin+log} i.e. 
\begin{align}\label{eq:expansion u,v}
u(z)&= \Re(\frac{a}{z}) + \alpha \ln\abs{z} + \mathsf{u}(z)\\
v(z)&= \Re(\frac{b}{z}) + \beta \ln\abs{z} + \mathsf{v}(z)
\end{align}
%where $\alpha,\beta\in\R$, $a,b \in \C$ and $\f,g C^{2,\alpha}(\overline{B}_\epsilon)$.
Furthermore we may assume that the support of $u$, $\operatorname{supp}(u)$, is disjoint to all the other ends. 
Then we have that
\begin{align}\label{eq:integration by parts1}
\int_\Sigma Lu\,v - u \, Lv \, d\mu_g= 2\pi( \alpha \mathsf{v}(0) - \mathsf{u}(0)\beta ) - 2\pi \Re(a \partial_z \mathsf{v}(0) - \partial_z \mathsf{u}(0)b).
\end{align}
\end{lemma}
\begin{proof}
First we note that $Lu\,v - u \, Lv= \Delta_gu \,v - u\,\Delta_gv$ and 
\[ -\int_\Sigma Lu\,v - u \, Lv \, d\mu_g= -\lim_{\epsilon \to 0} \int_{\Sigma\setminus D_\epsilon(p)} Lu\,v - u \, Lv \,d\mu_g= \lim_{\epsilon\to 0 } \int_{\partial D_{\epsilon}(p) } \frac{\partial u}{\partial \nu } v -  u\frac{\partial v}{\partial \nu }\;   \]
where $\nu$ denotes the outer normal to $\partial D_{\epsilon}(p)$. 
As mentioned before the operators $L$ and $\Delta_g$ are conformally invariant. Therefore, we may perform the calculations with respect to the metric $\delta_{\C}$ of the chart. Thus we get that
\begin{align*}
	&\int_{\partial D_\epsilon(p)} \frac{\partial u}{\partial \nu } v -  u\frac{\partial v}{\partial \nu }= \int_{\partial{B_\epsilon}} \frac{\partial u}{\partial r } v -  u\frac{\partial v}{\partial r } \\
	=& \int_{\partial B_\epsilon} \left(\Re(- \frac{a}{rz}) + \frac{\alpha}{r} + \partial_r\mathsf{u} \right)\left(\Re(\frac{b}{z}) + \beta \ln(r) + \mathsf{v}\right)\\
	 &- \left(\Re(\frac{a}{z}) + \alpha \ln(r) + \mathsf{u} \right)\left(-\Re(\frac{b}{rz}) + \frac{\beta}{r}  + \partial_r\mathsf{v}\right).
\end{align*}
Using $\int_{\partial D_\epsilon} \frac{\eta(r)}{z} =0$ for any function $\eta$ that only depends on $r$ we are left with 
\begin{align*}
	&\int_{\partial{B_\epsilon}} \frac{\partial u}{\partial r } v -  u\frac{\partial v}{\partial r } \\
	=& \int_{\partial B_\epsilon} \left(\Re(- \frac{a}{rz})\mathsf{v} - \mathsf{u}\Re(- \frac{b}{rz})\right) + \left(\frac{\alpha}{r}\mathsf{v} - \mathsf{u}\frac{\beta}{r}\right) + \left(\partial_r \mathsf{u} \mathsf{v} - \mathsf{u}\partial_r \mathsf{v} \right)\\
	=& I_\epsilon + II_\epsilon + III_\epsilon.
\end{align*}
The integrand in the last term is clearly bounded which implies $III_\epsilon = O(\epsilon)$. Since $\mathsf{v}=\mathsf{v(0)}+ O(z), \mathsf{u}=\mathsf{u(0)}+ O(z)$ we conclude that 
\[ II_\epsilon = 2\pi \left( \alpha \mathsf{v}(0) - \mathsf{u}(0) \beta \right) + O(\epsilon).\]
Finally we may expand $\mathsf{v}(z)= \mathsf{v}(0) + \partial_z \mathsf{v}(0)z + \partial_{\bar{z}}\mathsf{v}(0) \bar{z} + O(z^2)$ and therefore using that $\bar{z} = \frac{r^2}{z}$ 
\begin{align*} \int_{\partial B_\epsilon} \frac{a}{rz}\mathsf{v} &= \int_{\partial B_\epsilon} \frac{a \mathsf{v}(0)}{rz} + \frac{a \partial_z\mathsf{v}(0)}{r} + \frac{a \partial_{\bar{z}}\mathsf{v}(0)r}{z^2} + O(1)\\
&= 0 + 2\pi a \partial_z \mathsf{v}(0) + 0 + O(\epsilon). 
\end{align*}
Together with the equivalent expansion for $\mathsf{u}$ we conclude the lemma, since 
\[ I_\epsilon = - 2\pi \Re\left(a \partial_z \mathsf{v}(0) - \partial_z \mathsf{u}(0) b\right)+ O(\epsilon).\]
\end{proof}

This lemma suggests to introduce the following linear bounded operators on $E$, (compare \cite[section 5]{PerezRos}):
Given a function $u \in E$ with $u = \sum_i \Re(a_i h_i) + \sum_i \alpha_i l_i+ \mathsf{u}$, $\mathsf{u} \in C^{2,\alpha}(\Sigma)$, we set 
\begin{align*} \mathsf{L}(u) &:= \alpha \in \R^m & \mathsf{S}(u)&:= a\in \C^m \\
\mathsf{H}(u)&:=(\mathsf{u}(p_1), \cdots , \mathsf{u}(p_m)) \in \R^m & \mathsf{Z}(u)&:=(\partial_z\mathsf{u}(p_1), \cdots , \partial_z\mathsf{u}(p_m)) \in \C^m \end{align*}
where we denoted by an abused of notation $\partial_z\mathsf{u}(p_i)$ the value $\partial_z\mathsf{u}(0)$ where $\partial_z$ is calculated with respect to the fixed conformal chart around $p_i$. Note first that we have $E_0 = E \cap \{ u \colon \mathsf{S}(u) = 0 \}$ and $E_1 = E\cap \{ u \colon \mathsf{L}(u) =0 \}$. 

Using a partition of unity subordinate to the open sets $D(p_i)$, $i=1, \dotsc, m$ and $\Sigma \setminus \{p_1, \dotsc, p_m\}$ we directly conclude from the previous lemma 

\begin{corollary}\label{cor:integration by parts}
Given $u,v \in E$, we have that
\begin{align*} &\int_\Sigma Lu\,v - u \, Lv \,d\mu_g= \int_{\Sigma} \hat{L}u\,v - u \, \hat{L}v \, d\mu_{\hat{g}}\\
&= 2\pi\left( \mathsf{L}(u)\cdot \mathsf{H}(v) - \mathsf{H}(u)\cdot \mathsf{L}(v)\right) -2\pi \Re \left( \mathsf{S}(u)\cdot \mathsf{Z}(v) - \mathsf{Z}(u)\cdot \mathsf{S}(v) \right)\,.
\end{align*}
\end{corollary}

We are now interested in the image of $\hat{L}: E_1 \to C^{0,\alpha}(\Sigma)$. For any linear subspace $F\subset C^{0,\alpha}(\Sigma)$ we denote with $F^\perp$ its $L^2$-orthogonal in $C^{0,\alpha}(\Sigma)$ with respect to the metric $\hat{g}$.  Note that by our choice of $\lambda$ we have that $E \subset L^2(\Sigma, \mu_{\hat{g}})$.  By classical functional analysis we have that if $F\subset C^{0,\alpha}(\Sigma)$ is finite dimensional, there exist bounded orthogonal projections onto $F$ and $F^\perp$. So we have that
\[ C^{0,\alpha}(\Sigma)= F\oplus F^\perp\,. \]
Furthermore we have $F^{\perp \perp}= F$. Note that if for another linear subspace $W$ we have $F^\perp \subset W$ then $W^\perp$ is finite dimensional and we have $C^{0,\alpha}(\Sigma)= W\oplus W^\perp$ and $W^{\perp\perp}=W$. 
We define the following sets
\begin{align*} K&:=\operatorname{ker}(\hat{L}|_{C^{2,\alpha}(\Sigma)})\\
K_2&:=\hat{L}(E)^\perp, \;K_0:=\hat{L}(E_0)^\perp, \;K_1:= \hat{L}(E_1)^\perp
 \end{align*}
By classical elliptic regularity theory on compact manifolds we know that $\hat{L}$ is a Fredholm operator and so $K$ is finite dimensional. Since $\hat{L}$ is self-adjoint we have additionally that $\hat{L}(C^{2,\alpha}(\Sigma))= K^\perp$. But this has the consequences that for $i=0,1$
\[ \hat{L}(E)\supset \hat{L}(E_i) \supset \hat{L}(C^{2,\alpha}(\Sigma))= K^\perp\,.\]
In particular this implies that $K_i, i=0,1,2$ are finite dimensional, $K_2\subset K_i \subset K$ for $i=0,1$ and 
\[ \hat{L}(E)=K_2^{\perp}, \hat{L}(E_0)=K_0^{\perp}, \hat{L}(E_1)=K_1^{\perp}. \]
Compare \cite[Lemma 5.2]{PerezRos} for some results on $K_0$. But unfortunately their dimensional bounds on $K_0$ are too weak for us to provide the existence of logarithmic growing Jacobi fields in general. Hence we are particularly interested in $\hat{L}(E_1)$ and therefore we would like to identify $K_1$. Recall that the components of $n_X$ are elements in $K$.  

Before we are able to state our dimensional estimate on $K_1$, we need to recall that 
\[ n_X: \Sigma \to \Sp^2 \]
is a holomorphic map. 
%(for the sake of readability we will drop the subindex $X$).
In particular, we can associate to $n_X$ its ramification divisor 
\begin{equation}\label{eq:ramification devisor} R(n_X)= r(q_1) q_1 + \dotsm + r(q_r) q_r\,\end{equation}
where $r(q_i)$ denotes the vanishing order of $\partial_zn_X$. After an appropriate rotation and an appropriate choice of local coordinates around $q_i$ that
\begin{equation}\label{eq:good coordinates for an ramification point} g(z)= z^{r(q_i)+1} \text{ for } P\circ g(z)=n_X(z)\end{equation}
where $P(w)=\frac{1}{1+\abs{w}^2}\left(w + \bar{w}, -i(w-\bar{w}), -1 + \abs{w}^2\right)$ is the inverse of the stereographic projection from the north pole.

\begin{proposition}
Under the above assumptions we have that
\begin{equation}\label{eq:characterization of K_1}
	K_1=\{ u \in K\colon \mathsf{Z}(u)=0 \}\,.
\end{equation}
In particular we have 
\begin{equation}\label{eq:n_X in K_1}
u(z):=a \cdot n_X(z) \in K_1 \text{ for each } a\in \R^3.
\end{equation}
If $\Sigma$ is a topological sphere we get that 
\begin{equation}\label{eq:dimension of K_1}
K_1=\{ a\cdot n_X \colon a \in \R^3 \}.
\end{equation}	
\end{proposition}
%Note that the estimate \eqref{eq:dimension of K_1} is sharp by \eqref{eq:n_X in K_1}.
\begin{proof}
By definition of $V_1$ the map
\[ \mathsf{S}: V_1 \to \C^m \]
is one-to-one, in particular $\mathsf{S}: E_1 \to \C^m$ is onto. As noted before, for each $v \in E_1$ we have $\mathsf{L}(v)=0$. Since $K_1=\hat{L}(E_1)^\perp \subset K=\ker(\hat{L}|_{C^{2,\alpha}(\Sigma)})$ we deduce for $u\in K_1$ by Corollary~\ref{cor:integration by parts} that 
\[ \Re\left(\mathsf{Z}(u)\cdot \mathsf{S}(v)\right) = 0 \text{ for all } v \in E_1;\]
	giving precisely \eqref{eq:characterization of K_1}.	
Since every end $p_i, i=1, \dotsc, m$, of $\Sigma$ must be a ramification point of $n_X$, see \eqref{eq:expansion Gauss curvature}, we have $r(p_i)\ge 1$ for $i=1, \dotsc, m$. This implies \eqref{eq:n_X in K_1} due to \eqref{eq:characterization of K_1}.

	Now we want to proof the precise description \eqref{eq:dimension of K_1}. First let us recall that if $Y:\Sigma \to \R^n$ is  a branched minimal immersion with finite total curvature (and therefore complete ends), we have 
\[ \int_{\Sigma} K \, d\mu_{g_Y}= 2\pi \bigl( \chi(\Sigma) - \sum_{q \in \text{ ends of } Y} (e(q)+1) + \sum_{p \in \text{ branch points of } Y} b(p) \bigr)\,; \]
where $g_Y=Y^\sharp \delta_{\R^3}$, $e(q)$ denotes the multiplicity of an end $q$ of $Y$ and $b(p)$ the order of a branch point $p$, compare \cite[Theorem 1]{Fang}.
On the the other hand we have $n_Y^\sharp \delta_{\mathcal{S}^2}= -K d\mu_{g_Y}$ and thus 
\[ - \int_\Sigma K \, d\mu_{g_Y}= 4\pi \deg(n_Y)\,. \]
Since $n_Y: \Sigma \to \mathcal{S}^2$ is a holomorphic map, it is a covering map with associated ramification devisor $R(n_Y)$, \eqref{eq:ramification devisor} and so we have by the Riemann-Hurwitz formula that 
\[ \chi(\Sigma) = \chi(\mathcal{S}^2) \deg(n_Y) - \abs{R}(n_Y) = 2 \deg(n_Y) - \abs{R}(n_Y),\]
where $\abs{R}(n_Y)= \sum_{i} r(q_i)$.
Combining all the above equations we obtain for any non-constant branched immersion $Y$ 
\begin{equation}\label{eq:admissibility condition between ramification and ends}
	-\abs{R}(n_Y)=2\chi(\Sigma) - \sum_{q \in \text{ ends of } Y} (e(q)+1) + B(Y) 
\end{equation}
where we used the abbreviation $B(Y)=\sum_{p \in \text{ branch points of } Y} b(p)$. For us it will only important that $B(Y)\ge 0$. \\

Now we will restrict ourselfs to the situation of $\Sigma = \mathcal{S}^2$ in which case \eqref{eq:admissibility condition between ramification and ends} reads
\begin{equation}\label{eq:admissibility condition between ramification and ends for spheres} \abs{R}(n_Y)=\sum_{q \in \text{ ends of } Y} (e(q)+1) - B(Y)- 4.\end{equation}
We denote by $M_0=\{p_1, \dotsc, p_m\}$ the set of ends of $X$. By $M_1=\{ p_i \in M_0 \colon r(p_i)>1\}$ the subsets where the ramification is bigger than $1$ and finally by $R_1:=\{ q_i \colon q_i \notin M_0\}$ the set of ramification points of $n_X$ that are not an end. 
Evaluating \eqref{eq:admissibility condition between ramification and ends for spheres} for $X$, where each end is embedded, gives
\begin{equation}\label{eq:admissiblity for X}
\abs{R}(n_X)=2 \abs{M_0}-4. 	
\end{equation}
As noted before $r(p_i)\ge 1$ for each $p_i\in M_0$ and $r(p_i)\ge 2$ for each $p_i \in M_1$. So the above implies that 
\begin{equation}\label{eq:estimate on R_1}
	\abs{R_1}+\abs{M_1} \le \abs{M_0}-4
\end{equation}

Suppose that $u \in K_1$ which implies that $u$ is a bounded solution of $Lu=0$ on $\Sigma \setminus \{ q_1, \dotsc, q_r \}$. Due to \cite[formula 3.6]{MontielRos}
\begin{equation}\label{eq:mapping to minimal surfaces} Y(u) = u n_X+ \frac{1}{\abs{\partial_z n_X}^2} \left(\partial_{\bar{z}} u \partial_z n_X + \partial_{z} u \partial_{\bar{z}} n_X\right)	
\end{equation}
defines a (possibly branched) minimal immersion with possible ends in $\{q_1, \cdots, q_r \}$ and Gauss map $n_X$. Suppose $Y(u)$ is not constant. From \eqref{eq:characterization of K_1} we have $\partial_zu(p_i)=0$ for all $p_i \in M_0$. But this implies that  we have 
\[ \abs{Y(z)} \le C \abs{z}^{1-r(p_i)} + O(1) \] in local conformal coordinates. In particular, $p_i$ is at most an end of multiplicity $r(p_i)-1$ and therefore $Y(u)$ does not have ends in $M_0\setminus M_1$. For each $q_i\in R_1$ we have just by the boundedness of $\partial_z u$ around $q_i$ that
\[ \abs{Y(z)}\le C \abs{z}^{-r(q_i)}. \]
In other words, $q_i$ is at most an end of multiplicity $r(q_i)$. Denoting by $M_Y$ the set of ends of $Y$ we estimate 
\begin{align*}
	&\sum_{q \in E_Y} (e(q)+1) - B(Y)- 4 \le \sum_{q \in M_Y\cap M_1} (e(q)+1) + \sum_{q \in M_Y \cap R_1} (e(q)+1)- 4\\
	&\le \sum_{q \in M_Y\cap M_1} r(q) + \sum_{q \in M_Y\cap R_1} r(q) + \abs{M_Y\cap R_1} -4\\
	&\le \sum_{i=1}^r r(q_i)-\abs{M_0}+ \abs{M_1} + \abs{M_Y\cap R_1} -4\\
	&=\abs{R}(n_X)-\abs{M_0}+\abs{M_1}+ \abs{M_Y\cap R_1} -4<\abs{R}(n_X)\,,
\end{align*}
where we used \eqref{eq:estimate on R_1}. This contradicts \eqref{eq:admissibility condition between ramification and ends for spheres}. Hence $Y(u)$ must be constant which implies by \cite[Proposition 2]{MontielRos} that $u=a \cdot n_X$ for some $a \in \R^3$. 

\end{proof}

\begin{proposition}\label{prop:space of logarithmic growing Jacobi fields on spheres}
	Let $\Sigma$ be a topological sphere. Then we have  
	\[ J_{\log}:=\{ u \in E \colon \mathsf{L}(u) \neq 0 \text{ and } Lu=0 \text{ on } \Sigma \setminus\{ p_1, \dotsc, p_m \} \} \]
	is linearly isomorphic to $\R^N\setminus \{0\}$ with $$N= m- \dim\operatorname{span}\{ n_X(p_i) \colon i=1, \dotsc, m \}.$$ 
\end{proposition}
\begin{proof}
Let us define the linear map $A: \R^m \to \R^3$ given by
\[ A\alpha = \sum_{i=1}^m \alpha_i n_X(p_i).\]
Note that $\ker A \cong \R^N$. Hence the proposition follows if we show that $\mathsf{L}: J_{\log} \to \ker A\setminus\{0\}$ is a bijection. 

%First let us show that $\dim(J_{\log})\le m-3$.
For the inclusion $\mathsf{L}(J_{\log})\subset \ker(A)$ recall that for each $a \in \R^3$ we have that $v_a:=a\cdot n_X$ is an element of $K_1$ i.e.\ $\mathsf{Z}(v_a)=0$. Since $v_a$ is bounded, we have as well that $\mathsf{L}(v_a)=0, \mathsf{S}(v_a)=0$. Now given any element $u \in J_{\log}$, we apply Corollary~\ref{cor:integration by parts} to the functions $v_a$ and $u$ and obtain  
\[ 0 = 2\pi \mathsf{H}(v_a)\cdot \mathsf{L}(u). \]
If $\mathsf{L}(u)=\alpha$ this is equivalent to 
\[ a\cdot \left(\sum_{i=1}^m  n_X(p_i)\alpha_i\right)=0 \text{ for all } a \in \R^3. \]
Hence the inclusion follows.
%In particular we have that $J_{\log} \subset \ker{(A\circ L(\cdot))}$ where $A: \R^m \to \R^3$ is the linear map $A\alpha = \sum_{i=1}^m \alpha_i n_X(p_i)$ for $\alpha=(\alpha_1,\dotsc,\alpha_m)\in \R^m$.\\

It remains to show that $\mathsf{L}: J_{\log} \to \ker{(A)}\setminus \{0\}$ is onto. Let $\alpha \in \ker{(A)}\setminus\{0\}$ i.e.\ $\sum_{i=1}^m \alpha_i n_X(p_i)=0$. Consider the function $w= \sum_{i=1}^m \alpha_i l_i \in E$ where the $l_i$'s are the functions fixed in \eqref{eq:log functions}. So we have $\mathsf{L}(w)=\alpha$ and $\mathsf{S}(w)=0$. 
Applying Corollary~\ref{cor:integration by parts} to the function $w$ and $v_a$ we conclude for each $a \in \R^3$
\[ \int_\Sigma Lw\, v_a = 2\pi ( \mathsf{L}(w) \cdot \mathsf{H}(v_a) ) = 2\pi a  \cdot \sum_{i=1}^m \alpha_i n_X(p_i) = 0. \]
Hence by \eqref{eq:dimension of K_1} we conclude that $Lw \in K_1^\perp$. But then there is $v \in E_1$ with $Lv = Lw$. Since $\mathsf{L}(v)=0$ by the definition of $E_1$ the function $u=w-v$ satisfies
\[ Lu=0 \text{ and } \mathsf{L}(u)=\mathsf{L}(w)=\alpha\,.\]
This concludes the proof.
\end{proof}

\begin{lemma}\label{lem.X replacement function}
Let $\Sigma$, $J_{log}$ as above, then for each $u \in J_{log}$ there is a $\chi$ such that $\frac{\chi}{\abs{X}^2} \in W^{2,2}(\Sigma, d\mu_{\hat{g}})$  with the additional properties that 
\begin{itemize}
\item $\abs{\Delta_g \chi}^2 =16$ on $\Sigma$;
\item around each end $p_i$ it has in local conformal coordinates an expansion of the form \[ \chi(z) = \sigma_i \abs{X(z)}^2 + \mathsf{x}(z)\] for some $\sigma_i \in \{-1,+1\}$ and some smooth $\mathsf{x}$;
\item if $\vec{\sigma}=(\sigma_1, \dotsc, \sigma_m)$ then $\vec{\sigma}\cdot \mathsf{L}(u) \neq 0$.
\end{itemize}
\end{lemma}

\begin{proof}
Given $u \in J_{\log}$ we set $\alpha= \mathsf{L}(u) \in \R^m$. If $\sum_{i=1}^m \alpha_i \neq 0$ we choose $\sigma_i=1$ for all $i$ and choose $\chi = \abs{X}^2$. \\
Now suppose $\sum_{i=1}^m \alpha_i = 0$. After relabelling the ends we may assume without loss of generality that $\alpha_1 \neq 0$. Setting $\sigma_1 =-1$ and $\sigma_i =1$ for $i>1$ we have that $\sum_{i=1}^m \sigma_i \alpha_i = -2\alpha_1 \neq 0$. 
Fix a non-negative cut-off function $\eta$ on $D_\epsilon(p_1)$ such that $\operatorname{supp}(\eta)\subset D_\epsilon(p_1)$ and $\eta=1$ on $D_{\frac\epsilon2}(p_1)$ where $D_\epsilon(p_1)$ is the previously fixed domain of the conformal chart. We define $w:=\abs{X}^2-2\eta \abs{X}^2$. We have $w=\abs{X}^2$ outside of $D_{\epsilon}(p_1)$ and $w=-\abs{X}^2$ inside of $D_{\frac\epsilon2}(p_1)$. 
Furthermore, we have for some smooth function $r$
\[ \Delta_g w =\begin{cases}
4 &\text{ on } D_\epsilon(p_1)^c \\
r & \text{ on } D_\epsilon(p_1)\setminus D_{\frac\epsilon2}(p_1)\\
-4 &\text{ on } D_\epsilon(p_1)
\end{cases}\,.\]
For $s(x)= 4(1-2\, \mathbf{1}_{D_{\frac{4\epsilon}{3}}(p_1)}(x))$ the function 
\[ f(x):= \Delta_g w(x) - s(x) \]
is bounded and compactly supported in the annulus $A:=D_\epsilon(p_1)\setminus D_{\frac{\epsilon}{2}}(p_1)$; and so is the function $\lambda^{-1} f$ where $\lambda$ is the conformal factor introduced at the beginning of this chapter.
 
By classical $L^2$-theory on compact manifolds and the Lax-Milgram theorem the operator $\Delta_{\hat{g}}= \frac{1}{\lambda}\Delta_g: W^{2,2}(\Sigma, d\mu_{\hat{g}}) \to L^{2}(\Sigma)$ is onto. 
Let $v \in L^2(\Sigma,d\mu_{\hat{g}})$ be the solution of 
\[ \Delta_{\hat{g}} w = \frac{1}{\lambda} f. \]
Since $v$ is harmonic outside of $A$, we have that $v$ is smooth outside the compact set $A$. In particular $v$ is smooth around each end. Hence 
\[ \chi(x):= w(x)-v(x) \]
has the desired properties, where we use Lemma~\ref{lem.sobolev space} to deduce that $\frac{\chi}{\abs{X}^2} \in W^{2,2}(\Sigma, d\mu_{\hat{g}})$.
\end{proof}

Recall that Bryant proved \cite{BryantDuality, Bryant} that all unbranched Willmore spheres in $\R^3$ come from inverted complete minimal surfaces with $m$ embedded planar ends. The easiest case is $m=1$ for the standard round sphere. The next possible number after that is $m=4$ \cite{Bryant}.

\begin{corollary}\label{cor:index > 0 on spheres}
	Let $\Sigma$ be a topological sphere and $m \ge 4$ then 
	\[ \Ind(\Psi)=m-\dim \operatorname{span}\{ n_X(p_i) \colon i=1, \dotsc, m \} . \]
\end{corollary}
\begin{proof}
Since the upper bound was already proven in Corollary~\ref{cor:minusd}, it remains to show the lower bound. 
Recalling Definition~\ref{def:index} we let 
\[ Y:=\otimes_{\lambda<0} Y_\lambda  \text{ and }Y_0= \ker(Z)\,.\]
Assume by contradiction that $\dim(Y)< N:=m-\operatorname{span}\{ n_X(p_i) \colon i=1, \dotsc, m \}$.
In this case we can find $u \in J_{\log}$ such that $\frac{u}{\abs{X}^2} \in Y^\perp$, the $L^2$ orthogonal of $Y$ with respect to $\hat{g}$. After scaling of $u$ we may assume that $\int_{\Sigma} \frac{\abs{u}^2}{\abs{X}^4} \, d\mu_{\hat{g}} =1$. Let $\chi$ be the associated function constructed in Lemma~\ref{lem.X replacement function}.
 %Changing from $u$ to $-u$ we may assume that $\vec{\sigma} \cdot \mathsf{L}(u) >0$. 
We may take an $L^2$-orthogonal decomposition of $\frac{\chi}{\abs{X}^2}$ and $\frac{u}{\abs{X}^2}$ with respect to $Y, Y_0$ and their complement i.e. 
\[ \frac{\chi}{\abs{X}^2}= v + k_0 + r_0 \text{ and } \frac{u}{\abs{X}^2} = k_1 + r_1 \]
with $v \in Y$, $k_0, k_1 \in Y_0$ and $r_0, r_1 \in (Y\oplus Y_0)^\perp$. 
% It is essentially the same argument as in Theorem~\ref{thm:logJac implies index}. Let $u \in J_{\log}$ be non-zero and and $\chi$ the associated function constructed in Lemma~\ref{lem.X replacement function}. 
Now let us consider the vector field $w_t:= \chi + t u $ and the associated vectorfield 
\begin{equation}\label{eq:orthogonal decomposition} \psi_t=\frac{w_t}{\abs{X}^2}= v + (k_0+ tk_1) + (r_0+tr_1)=: v+ k_t + r_t.\end{equation}
By construction we have that $k_t \in Y$ and $r_t \in (Y\oplus Y_0)^\perp$ for all $t$. \\
Let us now evaluate the index form $\delta^2\mathcal{W}(\Psi)(\cdot, \cdot)$ in two ways:

First $w_t$ satisfies the assumptions of Theorem~\ref{thm:2nd var well-def new} for each $t\in \R$. Since $(v_1, \dotsc, v_m)=\vec{\sigma}$ and $(\beta_1, \dotsc , \beta_m)=t\,\mathsf{L}(u)$, we deduce that
\begin{align*}
2\delta^2\W(\Psi)(\psi_t,\psi_t)&= \int_{\Sigma\setminus \bigcup_{i=1}^m D_{\epsilon}(p_i)} (Lw_t)^2 - 4 \int_{\bigcup_{i=1}^m \partial D_{\epsilon}(p_i)} \frac{\partial \abs{X}^2}{\partial \nu} +  16\pi t \;\vec{\sigma} \cdot \mathsf{L}(u) + R_\epsilon \\
&= \int_{\Sigma\setminus \bigcup_{i=1}^m D_{\epsilon}(p_i)} (Lw_t)^2 - 4 \Delta_g \abs{X}^2 +  16\pi t\; \vec{\sigma} \cdot \mathsf{L}(u) + R_\epsilon\\
& = \int_{\Sigma\setminus \bigcup_{i=1}^m D_{\epsilon}(p_i)} (Lw_t)^2 - \abs{\Delta_g \chi}^2  +  16\pi t\; \vec{\sigma} \cdot \mathsf{L}(u) + R_\epsilon\\
&=\int_{\Sigma\setminus \bigcup_{i=1}^m D_{\epsilon}(p_i)} (2\Delta_g \chi + \abs{\nabla n_X}^2_g\chi)(\abs{\nabla n_X}^2_g\chi)   + 16\pi t \;\vec{\sigma} \cdot \mathsf{L}(u) + R_\epsilon.
\end{align*}
Since $\abs{\nabla n_X}^2_g=-2K_g$ is decaying sufficiently fast, compare  \eqref{eq:expansion Gauss curvature}, and $\Delta_g\chi$ is bounded on $\Sigma$, we can pass to the limit in $\epsilon$ independent of $t$ i.e. 
\begin{align}\nonumber 2\delta^2\W(\Psi)(\psi_t,\psi_t)&=\int_{\Sigma} (2\Delta_g \chi + \abs{\nabla n_X}^2_g\chi)(\abs{\nabla n_X}^2_g \chi) +16\pi t \;\vec{\sigma} \cdot \mathsf{L}(u)\\&= c_\chi +16 \pi t \;\vec{\sigma} \cdot \mathsf{L}(u), \label{eq:d^2W expansion 1} \end{align}
for some constant $c_\chi$. 

On the other hand we may use the orthogonal decomposition \eqref{eq:orthogonal decomposition} to evaluate $\delta^2\mathcal{W}(\Psi)(\cdot, \cdot)$ as a bilinear form on $W^{2,2}(\Sigma, d\mu_{\hat{g}})$: 
\begin{align}\nonumber
	&\delta^2\W(\Psi)(\psi_t,\psi_t)=\delta^2\W(\Psi)(v+r_t,v+r_t)\\\nonumber &= \delta^2\W(\Psi)(v,v)+ 2 \delta^2\W(\Psi)(v,r_t)+ \delta^2\W(\Psi)(r_t,r_t)\\\nonumber
	&=\delta^2\W(\Psi)(v,v)+ \delta^2\W(\Psi)(r_t,r_t)\\
	&\ge \lambda_{\min} \int_{\Sigma} \abs{v}^2 \, d\mu_{\hat{g}} + \delta^2\W(\Psi)(r_t,r_t)= -c_v + \delta^2\W(\Psi)(r_t,r_t) \label{eq:d^2W expansion 2}
\end{align}
for some constant $c_v>0$. In the estimates we used that $ \delta^2\W(\Psi)(k_t,\cdot)=0$ since $k_t \in Y_0$ and $ \delta^2\W(\Psi)(v,r_t)=\int_{\Sigma} Z(v) r_t \, d\mu_{\hat{g}} = 0$ as $Z(Y) \subset Y$.  
 Combining \eqref{eq:d^2W expansion 1} and \eqref{eq:d^2W expansion 2} we have 
 \[ \delta^2\W(\Psi)(r_t,r_t) \le c_\chi + c_v  +8 \pi t \;\vec{\sigma} \cdot \mathsf{L}(u)\,. \]
Since $\vec{\sigma} \cdot L(u) \neq 0$ for an appropriate choice of $t\in \R$ we get that $ \delta^2\W(\Psi)(r_t,r_t)<0$ for $r_t \in Y^\perp$. This is a contradiction.  
\end{proof}
 As a consequence we have for the Morin surface that $\Ind{\Psi}=1$. We can use it to investigate the symmetry along the associated variation. Hence the following proposition can be seen as an extension of Proposition~\ref{prop:symm}.

\begin{proposition} \label{prop:MorinSymmetry}
Let $\Psi: \Sp^2 \to \R^3$ be a the Morin immersion. There is locally a $C^1$-family of sphere immersions $\Psi_t: \Sp^2 \to \R^3$ such that $\Psi_0=\Psi$ and that locally decrease the Willmore energy and have a two fold symmetry. 
\end{proposition}
\begin{proof}
Recall that combining Corollary~\ref{cor:index > 0 on spheres} with Theorem~\ref{thm:forspheres} we deduce that $\Ind(\Psi)=1$. Hence let $u\in C^\infty(\Sp^2)$ be the function that realizes the negative Index i.e.
\[ Z u = - \lambda_1 u \]
for some $\lambda_1>0$. We may normalize $u$ such that $\int_{\Sp^2} \abs{u}^2 \, d\mu_{\hat{g}} =1$. 
Let $I: \Sp^2 \to \Sp^2$ be the orientation reversing isometry that is induced by the four fold orientation reversing symmetry of the Morin surface i.e. 
\[ S\circ\Psi(x)= \Psi\circ I(x) \text{ and } S(n(x))= - n\circ I(x) \]
where $S\in SO(3)$ that rotation by $\frac{\pi}{2}$.   

 Since $I$ is an isometry and $Z$ is not sensitive on the orientation we have that
\[ Z(u\circ I) = (Zu)\circ I = - \lambda_1 u\circ I \,.\]
Hence $u\circ I$ is as well an eigenfunction to the same eigenvalue. We conclude that $u \circ I = \sigma u$ for some $\sigma \in \R$. Again, since $I$ is an isometry, we have that 
\[ \sigma^2 \int_{\Sp^2} \abs{u}^2 \, d\mu_{\hat{g}} = \int_{\Sp^2} \abs{u\circ I}^2 \, d\mu_{\hat{g}}= \int_{\Sp^2} \abs{u}^2 \, d\mu_{I^\sharp \hat{g}}= \int_{\Sp^2} \abs{u}^2 \, d\mu_{\hat{g}}.  \]
Thus we conclude that $\sigma \in \{+1, -1\}$. 
Consider the family $\Psi_t = \Psi + t \, u n_\Psi$. This family realizes the Index i.e.\ it degreases locally the Willmore energy. Due to Proposition~\ref{prop:symm} we know that $\Psi_t$ can not preserve the four fold symmetry. Nonetheless we have 
\begin{align*}
S(\Psi_t(x))&= S(\Psi(x)) + t u(x) S(n(x))= \Psi(I(x))- t u(x)	n(I(x))\\
&= \Psi(I(x)) - t \sigma u(I(x)) n(I(x))\,.
\end{align*}
Hence we conclude on the one hand that $\sigma = +1$ since otherwise we would have $S(\Psi_t(x))=\Psi_t(I(x))$. On the other hand we have nonetheless $S^2(\Psi_t(x))=\Psi_t(I^2(x))$. This shows the two fold symmetry.
\end{proof}

\appendix
\section{Bounded harmonic functions are constant}
\begin{lemma}\label{lem:bounded harmonic functions are constant}
Let $X: \Sigma \to \R^n$ be a complete minimal immersion  with finite total curvature $K$. If $u$ is bounded harmonic function on $\Sigma$ then $u$ is constant. 	
\end{lemma}

\begin{proof}
The statement follows from a point singularity removability theorem for harmonic functions. 
Recall we can choose a local complex coordinate on a neighbourhood $D(p_i)$ around and end $p_i$ in $\Sigma$ such that $z(p_i)=0$. With respect to this local coordinates, we have 
\[ X_z(z) = - \frac{a_i}{z^{m+1}} + Y(z) \]
where the multiplicity of the end $m\ge 1$, because $X$ is complete, and $\abs{z^{m+1} Y(z)} \to 0$ as $z \to 0$. In particular we have 
\[ \abs{X_z}^2 = \abs{z}^{-2m-2} (\abs{a_i}^2 + b_i(z)) \]
with $b_i(z)=O(z)$ smooth. Without loss of generality we may even assume that $\abs{a_i}^2=2$. 
To ``compactify'' the Laplacian on $\Sigma$ we fix a function $\lambda \in C^\infty(\Sigma)$ strictly positive such that $\lambda(z) = \frac{1}{2} \left( \abs{z}^{-2m-2} (\abs{a_i}^2 + b_i(z)) \right)$ around each end $p_i$. The metric $\overline{g}:= \frac{1}{\lambda} g$ is smooth and compatible with the complex structure on $\Sigma$ determined by the pullback metric $g = X^\#\delta_{\R^n}$.  Hence we have $\Delta_{\overline{g}}= \lambda \Delta_{g}$. Furthermore $(\Sigma, \overline{g})$ is now compact. 
Recall that $u$ is harmonic hence in the local coordinates around the end $p_i$ we have that
\[ \Delta u(z) =0  \text{ on } B_{\epsilon} \setminus \{0\}. \]
Since $u$ is bounded, $u$ is harmonic across $0$ and we have 
\[ \Delta_{\overline{g}} u = 0 \text{ on } \Sigma . \]
But every harmonic function on a compact manifold is constant.  
\end{proof}

\bibliographystyle{plain}
\bibliography{Lit_2}

\end{document}